\documentclass[12pt,a4paper,twoside]{amsart}

\usepackage[english]{babel}

\usepackage[utf8]{inputenc}
\usepackage[T1]{fontenc}
\usepackage[cm]{aeguill}
\usepackage{times} 

\usepackage{amsmath,amsfonts,amsthm,amssymb} 
\usepackage{bbm}
\usepackage{latexsym}
\usepackage{mathtools}
\usepackage[all]{xy}

\usepackage{ifpdf}
\ifpdf
 \usepackage[pdftex,a4paper]{geometry}
 \usepackage[pdftex]{graphicx}
\else
 \usepackage[dvips,a4paper]{geometry}
 \usepackage[dvips]{graphicx}
 \renewcommand{\includegraphics}[1]{[#1]}
\fi
\ifpdf
 \usepackage[colorlinks=true,urlcolor=black,linkcolor=black,citecolor=black]{hyperref}
\else
 
\fi

\newcommand{\alias}[2]{
 \providecommand{#1}{}
 \renewcommand{#1}{#2}
}

\newcommand{\macro}[3]{
 \providecommand{#1}{}
 \renewcommand{#1}[#2]{#3}
}

\alias{\P}{\mathcal{P}}
\alias{\M}{\mathcal{M}}
\alias{\W}{\mathcal{W}}
\alias{\O}{\mathcal{O}}

\alias{\N}{\mathbb{N}}
\alias{\Z}{\mathbb{Z}}
\alias{\Q}{\mathbb{Q}}
\alias{\R}{\mathbb{R}}
\alias{\C}{\mathbb{C}}
\alias{\S}{\mathbb{S}}
\alias{\H}{\mathbb{H}}
\macro{\fn}{3}{#1:#2\to#3}
\macro{\fun}{3}{#1:#2\mapsto#3}
\macro{\func}{5}{#1:\begin{array}[t]{rcl}
#2 & \rightarrow & #3\\
#4 & \mapsto     & #5\end{array}}

\macro{\name}{1}{{\it #1}}
\alias{\(}{\left(}
\alias{\)}{\right)}
\macro{\sing}{1}{\left\{ #1 \right\}}
\macro{\set}{2}{\left\{ #1 \,\middle|\, #2 \right\}}

\macro{\wt}{1}{\widetilde{#1}}

\macro{\eqpunct}{1}{\, #1}

\alias{\Pr}{\textbf{P}_\Gamma}
\alias{\id}{\text{id}}
\alias{\Conv}{\text{Conv}}
\alias{\card}{\#}
\alias{\diam}{\text{diam}\,}
\alias{\supp}{\text{supp}\,}
\alias{\SO}{\text{SO}}
\alias{\Diag}{\sing{\text{diagonal}}}

\macro{\mtext}{1}{\;\text{#1}\;}
\alias{\eqand}{\mtext{and}}
\alias{\eqor}{\mtext{or}}

\newcommand{\itv}[3][cc]{
 \def\itvcc{cc}%
 \def\itvco{co}%
 \def\itvoc{oc}%
 \def\itvoo{oo}%
 \def\tmparg{#1}%
 \ifx\tmparg\itvcc
  \left[ #2 ; #3 \right]
 \else
  \ifx\tmparg\itvco
   \left[ #2 ; #3 \right[
  \else
   \ifx\tmparg\itvoc
    \left] #2 ; #3 \right]
   \else
    \ifx\tmparg\itvoo
     \left] #2 ; #3 \right[
    \fi
   \fi
  \fi
 \fi
}

\theoremstyle{plain}
\newtheorem{theo}{Theorem}
\newtheorem*{theo*}{Theorem}
\newtheorem*{theo.recite1}{Theorem \ref{thm:erg}}
\newtheorem{lemm}{Lemma}[section]
\newtheorem{prop}[lemm]{Proposition}

\theoremstyle{definition}
\newtheorem{defi}{Definition}

\title[Finiteness of Gibbs measures]{Finiteness of Gibbs measures on noncompact manifolds with pinched negative curvature}

\author{Vincent Pit}
\address[Vincent Pit]{UFRJ, Instituto de Matem\'atica, Av. Athos da Silveira Ramos 149 \\ Centro de Tecnologia - Bloco C - Cidade Universit\'aria - Ilha do Fund\~ao \\ Rio de Janeiro, RJ, Brasil}
\email{pit@im.ufrj.br}

\author{Barbara Schapira}
\address[Barbara Schapira]{I.R.M.A.R. UMR CNRS 7352, 
UFR de math\'ematiques, Universit\'e Rennes 1
Campus de Beaulieu,  
263 avenue du G\'en\'eral Leclerc, CS 74205
35042 RENNES Cedex}
\email{barbara.schapira@univ-rennes1.fr}

\begin{document}

\begin{abstract}
 We characterize the finiteness of Gibbs measures for geodesic flows on negatively curved manifolds by several criteria, analogous to those proposed by Sarig for symbolic dynamical systems over an infinite alphabet.
 As an application, we recover Dal'bo-Otal-Peign\'e criterion of finiteness for the Bowen-Margulis measure on geometrically finite hyperbolic manifolds, as well as Peign\'e' examples of gemetrically infinite manifolds having a finite Bowen-Margulis measure.
 These criteria should be useful in the future to find more examples with finite Gibbs measures.
\end{abstract}

\maketitle



\section*{Introduction}

Hyperbolic dynamical systems are, heuristically, so chaotic that all behaviours that one can imagine indeed happen for some orbits.
From the point of view of ergodic theory, this can be expressed by the existence of Gibbs measures.
When the space is compact, choosing any regular enough "weight function", called a \name{potential}, i.e. usually a H\"older continuous map on the space, we can find an invariant ergodic probability measure which gives, roughly speaking, large measure to the sets where the potential is large, and small measure to those where the potential is small.
This is a quantified way of saying that all behaviours that one can imagine, represented by the choice of a potential, indeed happen for a  hyperbolic dynamical system.

Such measures are called \name{Gibbs measures}, and their existence for H\"older potentials and uniformly hyperbolic flows has been proved by Bowen-Ruelle \cite{BR}.
Thermodynamical formalism, i.e. the study of the existence and properties of these measures, has been extended to noncompact situations in two main cases.
In symbolic dynamics, Sarig (\cite{Sa99}, \cite{Sa01}) studied thermodynamical formalism of shifts over a countable alphabet.
In the context of geodesic flows of noncompact negatively curved manifolds, the measure of maximal entropy, which is associated with the constant potential, now commonly called the \name{Bowen-Margulis measure}, has been extensively studied, first by Sullivan on geometrically finite hyperbolic manifolds \cite{Su84}, and later with his ideas by many others.
Among them, let us cite Otal-Peign\'e who obtained an optimal variational principle in \cite{OP}, and Roblin \cite{Roblin} who established an equidistribution result of measures supported by periodic orbits towards the Bowen-Margulis measure.
Generalizations of their results to Gibbs measures with general H\"older potentials have been proved in \cite{PPS}.

All these results hold if and only if the Gibbs measure associated with the potential is finite.
It is therefore very important to be able to characterize, or at least to give sufficient conditions for the finiteness of Gibbs measures.
Some partial results and examples have already been proved in the past.
In \cite{Sa99}, \cite{Sa01}, in a symbolic context, Sarig established two finiteness criteria for the measure associated with a H\"older potential $F$.
Iommi and its collaborators  extended this study to suspension flows over such shifts, in \cite{Barreira-Iommi}, \cite{Iommi-Jordan-Todd}. In the case of geodesic flows on the unit tangent bundle of noncompact manifolds, the first criterion appeared in \cite{DOP}.
In the particular case of geometrically finite manifolds, Dal'bo-Otal-Peign\'e showed that the Bowen-Margulis measure is finite if and only if a series involving the parabolic elements of the fundamental group converges.
This criterion has been extended later by Coud\`ene \cite{Cou} to Gibbs measures of geometrically finite manifolds.
Finally, Peign\'e constructed in \cite{Pei} the first examples of geometrically infinite hyperbolic manifolds whose Bowen-Margulis measure is finite.
His proof, once again, involved the convergence of a certain series.
Ancona \cite{Ancona} also obtained such examples, but through harmonic analysis.

Our main motivation is to give a  unified way to check whether a Gibbs measure is finite, not specific to a certain class of manifolds, and allowing to recover all results mentioned above, the geometric ones as well as Sarig' symbolic criteria.
We will provide three equivalent criteria for the finiteness of Gibbs measures, all involving the convergence of some series, two of them being the geometric analogues of Sarig' criteria in terms of lengths of periodic orbits, the other one being a reformulation in terms of the action of the fundamental group $\Gamma$ of $M$ on its universal cover, which is more convenient in our geometric context.

To this end, we will not rely on having a symbolic coding for the geodesic flow, which cannot be ensured in the general case, but on geometrical estimates and \name{Kac's recurrence lemma}.
Recall that this lemma asserts that if a measure $\mu$ is conservative and $A$ is a measurable set with positive finite measure, then the measure of the whole space equals $\sum_{n \ge 1} n \mu(A_n)$, where $A_n$ denotes the subset of points of $A$ that return in $A$ after exactly $n$ iterations of the dynamics.
Therefore, the finiteness of $\mu$ is equivalent to the convergence of a certain series.
However, Kac Lemma is in general more an abstract result than an useful criterion, due to the difficulty to estimate the measure of the sets $A_n$, but the geometry of negatively curved manfolds will allow us to convert it into an explicit   efficient criterion.

Let us first give some notations in order to state our results.
We are interested in the geodesic flow $(g^t)$ on the unit tangent bundle $T^1 M$ of a complete manifold with pinched negative curvature.
Our results also hold when $M$ is a negatively curved orbifold, that is the quotient of a complete simply connected negatively curved manifold by a discrete nonelementary group $\Gamma$ which can contain torsion elements.
We study  this flow  in restriction to its \name{nonwandering set} $\Omega$.
We denote by $\P$ the set of periodic orbits, by $\P_\W$ the set of periodic orbits which intersect some set $\W \subset T^1 M$, and by $\P_\W'$ the set of primitive periodic orbits.
For a given periodic orbit $p \in \P_\W$, we denote by $l(p)$ its length, and by $n_\W(p)$ the "number of times that the geodesic $p$ crosses $\W$" (see section \ref{sec:numret} for a more precise definition).
We consider a potential $\fn{F}{T^1 M}{\R}$, i.e. a H\"older continuous map, and denote by $P(F)$ its pressure.
It is shown in \cite{PPS} how one can build a Gibbs measure $m_F$ associated with the potential $F$.

The Hopf-Tsuji-Sullivan theorem (see \cite[Theorem 5.4]{PPS}) asserts that a Gibbs measure $m_F$ is either ergodic and conservative (possibly finite or infinite) 
or totally dissipative, depending on whether the \name{Poincar\'e series} of $(\Gamma, F)$ is respectively divergent or convergent.
Therefore, before investigating the finiteness of a Gibbs measure, we investigate when it is ergodic and conservative.

\begin{defi}[Recurrence]
 \label{def:recurrent-potential}
 A potential $\fn{F}{T^1 M}{\R}$ is said to be \name{recurrent} when there exists an open relatively compact subset $\W$ of $T^1 M$, which intersects the nonwandering set $\Omega$, such that
 \[ \sum_{p \in \P} n_\W(p) e^{\int_p (F - P(F))} = +\infty \eqpunct{.} \]
\end{defi}

By analogy with the recurrence property for potentials on infinite subshifts developed in \cite{Sa01}, where a periodic orbit may be made of several periodic points beginning with the same letter, this integer $n_\W(p)$ can be interpreted as how many changes of origins are possible along the geodesic orbit $p$ so that the parameterization starts in $\W$.

Our first result is a reformulation of the divergence of the Poincar\'e series in terms of periodic orbits.

\begin{theo}[Ergodicity criterion]
 \label{thm:erg}
 Let $M$ be a negatively curved orbifold, and $\fn{F}{T^1 M}{\R}$ a H\"older continuous potential with $P(F) < +\infty$.
 Then the Gibbs measure $m_F$ is ergodic and conservative if and only if $F$ is recurrent.
\end{theo}

Unfortunately, this equivalence is unlikely to be very useful in practice.
The main interest of Theorem \ref{thm:erg} is to enlighten the very strong analogy between our results on geodesic flows on noncompact manifolds and Sarig's work in symbolic dynamics over a countable alphabet, despite the fact that no general coding result of the geodesic flow by a symbolic dynamical system is known in this context.

\begin{defi}[Positive recurrence for the geodesic flow]
 \label{def:positive-recurrent}
 Let $M$ be a negatively curved orbifold with pinched negative curvature.
 A H\"older continuous potential $\fn{F}{T^1 M}{\R}$ with $P(F) < +\infty$ is said \name{positive recurrent} relatively to a set $\W \subset T^1 M$ intersecting $\Omega$ and the integer $N \ge 1$ if
it is recurrent and
 \[ \sum_{\substack{p \in \P_\W' \\ n_\W(p) \le N}} l(p) e^{\int_p (F - P(F))} < +\infty \eqpunct{.} \]
\end{defi}

Our main finiteness criterion is the following result, which is the geometric analogue of the symbolic criterion of Sarig \cite{Sa01}.

\begin{theo}[First finiteness criterion]
 \label{thm:crit3}
 Let $M$ be a negatively curved orbifold, and $\fn{F}{T^1 M}{\R}$ a H\"older continuous potential with $P(F) < +\infty$.
 Denote by $m_F$ its associated Gibbs measure.
 \begin{enumerate}
  \item[$(i)$\:]  If $F$ is recurrent, and there exist an open relatively compact set $W \subset M$ meeting $\pi(\Omega)$ such that $F$ is positive recurrent with respect to $\W = T^1 W$ and some $N \ge K_\W$ (where $K_\W$ depends only on $\diam W$ and on the geometry of $M$), then $m_F$ is finite.
  \item[$(ii)$\:] If $m_F$ is finite, then $F$ is recurrent, and positive recurrent with respect to $\W = T^1 W$ and any $N \ge 1$ for every open relatively compact set $W \subset M$ meeting $\pi(\Omega)$.
 \end{enumerate}
\end{theo}

In particular, when $F$ is recurrent, then $F$ is positive recurrent with respect to $\W = T^1 W$ for some open relatively compact set $\wt{W}$ which intersects $\pi(\wt{\Omega})$ and for some $N \ge K_\W$ if and only if it is positive recurrent relatively to any such set.

The proof of Theorem \ref{thm:crit3} follows from Theorem \ref{thm:crit1} below, which expresses the finiteness of $m_F$ in terms of the action of $\Gamma$ on $\wt{M}$ instead of periodic orbits on $T^1 M$.
This criterion does not appear in Sarig's work because it has no meaning in a purely symbolic setting.
However, it is very useful in our geometrical context.

We start by introducing a notation.
Given a subset $\wt{W} \subset \wt{M}$, denote by
\[ \Gamma_{\wt{W}} = \set{\gamma \in \Gamma}{\exists y, y' \in \wt{W}, \itv[cc]{y}{\gamma y'} \cap g \wt{W} \neq \emptyset \Rightarrow \overline{\wt{W}} \cap g \overline{\wt{W}} \neq \emptyset \eqor \gamma \overline{\wt{W}} \cap g \overline{\wt{W}} \neq \emptyset} \]
the set of elements $\gamma$ such that there exists a geodesic starting from $\wt{W}$ and finishing in $\gamma \wt{W}$ that meets the orbit $\Gamma \wt{W}$ only at the beginning or at the end.

\begin{defi}[Positive recurrence in the universal cover]
 \label{def:pos-rec-bis}
 The pair $(\Gamma, \wt{F})$ is said to be \name{positive recurrent} with respect to a set $W \subset \wt{M}$ such that $\W = T^1 W$ intersects $\Omega$ if $F$ is recurrent and
 \[ \exists x \in \wt{M}, \sum_{\gamma \in \Gamma_{\wt{W}}} d(x, \gamma x) e^{\int_x^{\gamma x} (\wt{F} - P(F))} < +\infty \eqpunct{.} \]
\end{defi}

\begin{theo}[Second finiteness criterion]
 \label{thm:crit1}
 Let $M$ be a negatively curved orbifold with pinched negative sectional curvature.
 Let $\fn{F}{T^1 M}{\R}$ be a H\"older continuous potential with $P(F) < +\infty$, and denote by $m_F$ the associated Gibbs measure on $T^1 M$.
 \begin{enumerate}
  \item[$(i)$\:]  If $F$ is recurrent, and if $(\Gamma, \wt{F})$ is positive recurrent with respect to some open relatively compact set $\wt{W} \subset \wt{M}$ meeting $\pi(\wt{\Omega})$, then $m_F$ is finite.
  \item[$(ii)$\:] If $m_F$ is finite, then $F$ is recurrent, and $(\Gamma, \wt{F})$ is positive recurrent with respect to any open relatively compact set $\wt{W} \subset \wt{M}$ meeting $\pi(\wt{\Omega})$.
 \end{enumerate}
\end{theo}

In particular, when $F$ is recurrent, then $(\Gamma, \wt{F})$ is positive recurrent relatively to some open relatively compact set $\wt{W}$ which intersects $\pi(\wt{\Omega})$ if and only if it is positive recurrent relatively to any such set.

In \cite{Sa99}, Sarig proved a finiteness criterion, established earlier than the symbolic analogue of Theorem \ref{thm:crit3}.
This criterion seems less practical than his later work.
However, we wanted to show a complete analogy between the symbolic and our geometric settings, so we established the same criterion in our situation.
The proof is different from the previous criteria and relies on equidistribution of weighted periodic orbits.
This criterion requires the assumption that the geodesic flow is topologically mixing on $\Omega$.
This extremely classical assumption is satisfied in most interesting situations, even if its validity is open in general.

\begin{defi}
 The potential $F$ is said to be \name{positive recurrent in the first sense of Sarig \cite{Sa99}} if there exists an open relatively compact subset $\W$
 of $T^1 M$ meeting $\Omega$, and constants $c > 0$,   $t_0 \ge 0$ and   $C > 0$ such that
 \[ \forall t \ge t_0, \frac{1}{C} \le \sum_{\substack{p \in \P \\ t - c < l(p) \le t}} n_\W(p) e^{\int_p (F - P(F))} \le C \eqpunct{.} \]
\end{defi}

\begin{theo}[Third finiteness criterion]
 \label{thm:crit4}
 Let $M$ be a negatively curved complete orbifold, with pinched negative curvature.
 Assume that its geodesic flow is topologically mixing.
 Let $\fn{F}{T^1 M}{\R}$ be a H\"older continuous potential with $P(F) < +\infty$, and denote by $m_F$ its associated Gibbs measure on $T^1 M$.
 Then $m_F$ is finite if and only if $F$ is positive recurrent in the first sense of Sarig with respect to some open relatively compact set $\W$ meeting $\Omega$.
\end{theo}

When this theorem holds, $F$ is actually positive recurrent in the first sense of Sarig with respect to any open relatively compact set $\W$ meeting $\Omega$.

The structure of this paper goes at follows.
Section \ref{sec:prelims} introduces the geometric and thermodynamic formalism background, including some elementary lemmas of hyperbolic geometry stated in a convenient way for our purposes.
In particular, Lemma \ref{lem:lrsa} plays a crucial role in the proof.
Section \ref{sec:numret} introduces the notion of \name{number of returns of a periodic orbit}, which is used to state Theorem \ref{thm:erg}, \ref{thm:crit3} and \ref{thm:crit4}.
The proof of Theorem \ref{thm:erg} is given in section \ref{sec:erg}.
Theorem \ref{thm:crit1} is proved in section \ref{sec:fin1}, from which an intermediate Theorem \ref{thm:crit2} is derived in section \ref{sec:fin2}, and Theorem \ref{thm:crit3} is itself derived in section \ref{sec:fin3}.
In section \ref{sec:fin4}, we state and prove a couple of equidistribution results for nonprimitive periodic orbits, and derive from it the proof of Theorem \ref{thm:crit4}.
Finally, we show in section \ref{sec:examples} how to retrieve previous finiteness results from ours.

We believe that these criteria will lead to new examples of interesting manifolds with finite Gibbs measure.
This will be done in the future.


\section{Preliminaries}

\label{sec:prelims}


\subsection{Geodesic flow in negative curvature}

In the following, $\wt{M}$ is a Hadamard manifold with pinched negative sectional curvature $-b^2 \le k \le -a^2 < 0$, $\Gamma$ is a discrete group of isometries preserving orientation of $\wt{M}$, $M = \Gamma \backslash \wt{M}$ is the quotient orbifold (manifold whenever $\Gamma$ has no torsion elements), and $T^1 M = \Gamma \backslash T^1 \wt{M}$ is its unit tangent bundle.
Observe once and for all that Riemannian/differential concepts are still well defined on $M$ or $T^1 M$ by defining objects first on the universal cover, and then going down to $M$.
With a slight abuse of notation, we denote by $\fn{\pi}{T^1 M}{M}$ or $\fn{\pi}{T^1 \wt{M}}{\wt{M}}$ the canonical projection, and by $\fn{\Pr}{T^1 \wt{M}}{T^1 M}$ or $\fn{\Pr}{\wt{M}}{M}$ the quotient maps.

The geodesic flow of $T^1 \wt{M}$ and of $T^1 M$ is denoted by $(g^t)_{t \in \R}$.
The boundary at infinity $\partial_\infty \wt{M}$ is the set of equivalence classes of geodesic rays staying at bounded distance from each other.
If $v \in \wt{M}$, we denote by $v_\pm$ the positive and negative endpoints of the geodesic it defines.

The limit set $\Lambda(\Gamma)$ is the closure of any $\Gamma$ orbit of $\wt{M}$ in $\partial_\infty \wt{M}$.
We will only consider the nontrivial case where $\Gamma$ is nonelementary, that is $\Lambda(\Gamma)$ is infinite.
Eberlein proved that the nonwandering set $\Omega$ of the geodesic flow coincides with the set of vectors $v \in T^1 \wt{M}$ such that $v_\pm \in \Lambda(\Gamma)$.

The \name{Hopf coordinates} relatively to any base point $x_0 \in \wt{M}$ are given by
\[ v \in T^1 \wt{M} \mapsto (v_-, v_+, \tau_{x_0}(v)) \eqpunct{.} \]
where $\tau_{x_0}(v)$ is the algebraic distance on the geodesic $(v_- v_+)$ from the projection $p_{x_0}(v)$ of $x_0$ on this geodesic to the base point $\pi(v)$.
They induce an extremely useful homeomorphism between $T^1 \wt{M}$ and $\partial_\infty \wt{M} \times \partial_\infty \wt{M} \times \R$.
In these coordinates, the geodesic flow acts by translation on the real coordinate, so that all dynamically relevant sets can be expressed nicely in terms of these coordinates.

The set of periodic orbits (respectively primitive periodic orbits) of the geodesic flow is denoted by $\P$ (respectively $\P'$).
Recall that periodic orbits of the geodesic flow are in $1-1$ correspondance with conjugacy classes of hyperbolic elements of $\Gamma$.
More precisely, let $\Gamma_h$ be the set of hyperbolic (or loxodromic) isometries of $\Gamma$, and $\Gamma_h'$ those which are primitive.
By definition, such a $\gamma$ has two fixed points in $\partial_\infty \wt{M}$, one repulsive and the other attractive.
It acts by translation on the geodesic line of $\wt{M}$ joining them so that the geodesic orbit of $T^1 \wt{M}$ from the repulsive to the attractive endpoint induces on $T^1 M$ a periodic orbit of the geodesic flow.
We will denote by $l(\gamma)$ for $\gamma \in \Gamma_h$, or equivalently $l(p)$ for $p \in \P$, the period of this orbit.


\subsection{Thermodynamical formalism}

In this section, we recall briefly some facts about thermodynamical formalism on negatively curved manifolds, which are either classical or can be found in \cite{PPS}.

Let $\fn{F}{T^1 M}{\R}$ be a H\"older continuous map (or \name{potential}), and $\wt{F}$ be its $\Gamma$-invariant lift to $T^1 \wt{M}$.
Its \name{topological pressure} is defined as the supremum
\[ P(F) = \sup_{\mu \in \M} \( h(\mu) + \int F d\mu \) \eqpunct{,} \]
the supremum being taken over the set $\M$ of all invariant probability measures, and $h(\mu)$ being the Kolmogorov-Sinai entropy of $\mu$.

A \name{dynamical ball} $B(v, T, \varepsilon)$ is the set
\[ B(v, T, \varepsilon) = \set{w \in T^1 M}{d(\pi(g^t(v)), \pi(g^t(w))) \le \varepsilon \mtext{for all} 0 \le t \le T} \eqpunct{.} \]
Here, there is a slight abuse of notation : we use the distance $d$ on $M$ instead of a distance on $T^1 M$, for example the Sasaki metric.
However, standard results about geodesic flows in negative curvature show that these two points of views are equivalent.
We refer to \cite{PPS} for details.
An invariant Radon measure $\mu$ satisfies the \name{Gibbs property} (see \cite[Section 3.8]{PPS}) if for all compact subsets $K \subset T^1 M$, there exists a constant $C_K > 0$, such that for all $v \in K$ and $T > 0$ such that $g^T(v) \in K$, one has
\begin{equation} \label{Gibbs}
 \frac{1}{C_K} e^{\int_0^T F(g^t(v)) dt - T P(F)} \le \mu(B(v, T, \varepsilon) \le C_K e^{\int_0^T F(g^t(v)) dt - T P(F)} \eqpunct{.}
\end{equation}
The careful reader will observe that this definition is slightly simplified compared to \cite{PPS}, but describes the same measures.

When $P(F)$ is finite, the Patterson-Sullivan-Gibbs construction, detailed in \cite{PPS}, allows to build a measure $m_F$, which satisfies the above Gibbs property, whose lift $\tilde{m}_F$ on $T^1 \wt{M}$ has the following nice expression in the Hopf coordinates
\begin{equation}
 \label{eqn:gibbsdef}
 d\tilde{m}_F(v) = \frac{1}{D_{F-P(F),x_0}(v_-, v_+)^2} d\mu_{x_0}^{F \circ \iota}(v_-) d\mu_{x_0}^F(v_+) dt \eqpunct{,}
\end{equation}
where $\iota$ is the flip map $\fn{\iota}{v}{-v}$, $\mu_{x_0}^F$ is the so-called \name{Patterson-Sullivan-Gibbs conformal density} on the boundary, and $D_F$ is the \name{$F$-gap map from $x$} defined as
\[ D_{F,x}(\xi, \eta) = \exp \frac{1}{2} \( \lim_{t \to +\infty} \int_x^{\eta_t} \wt{F} - \int_{\xi_t}^{\eta_t} \wt{F} + \int_{\xi_t}^x \wt{F} \) \eqpunct{.} \]
This map is continuous and positive on $\partial_\infty \wt{M} \times \partial_\infty \wt{M} \setminus \Diag$, and therefore bounded away from $0$ and $+\infty$ on all compact sets of $\partial_\infty \wt{M} \times \partial_\infty \wt{M} \setminus \Diag$.
Moreover, the point $x_0$ being arbitrary, the Patterson-Sullivan-Gibbs conformal densities $(\mu_x^F)_{x \in \wt{M}}$ form a family of measures that have full support in $\Lambda(\Gamma)$ and are $\Gamma$-quasi-invariant.

The Gibbs measure $m_F$ satisfies the following alternative, known as the Hopf-Tsuji-Sullivan Theorem, proved by Roblin \cite{Roblin} in full generality when $F\equiv 0$, and
whose proof has been adapted to Gibbs measures in \cite{PPS}.
First, recall that the pressure $P(F)$ is also the critical exponent of the following \name{Poincar\'e series}
\[ P_{\Gamma,x,F}(s) = \sum_{\gamma \in \Gamma} e^{\int_x^{\gamma x} (\wt{F} - s)} \eqpunct{.} \]
This Hopf-Tsuji-Sullivan-Roblin theorem for Gibbs measures 
enlights how the convergence or divergence of the above series for $s = P(F)$ is a crucial point for the ergodicity of the Gibbs measure $m_F$.

\begin{theo*} 
 Let $M$ be a negatively curved orbifold with pinched negative curvature, and $\fn{F}{T^1 M}{\R}$ a H\"older continuous map with finite pressure.
 Then the measure $m_F$ is ergodic and conservative if and only if the Poincar\'e series $P_{\Gamma,x,F}(s)$ diverges at $s = P(F)$, i.e.
 \[ \sum_{\gamma \in \Gamma} e^{\int_{x_0}^{\gamma x_0} (\wt{F} - P(F))} = +\infty \eqpunct{,} \]
 and the measure $m_F$ is totally dissipative otherwise.
\end{theo*}

In fact, one can show that "the" measure $m_F$ built in \cite{PPS} is well defined if and only if the above series diverges.
As this is the only interesting case for us, we do not care about this problem of terminology.

When a Gibbs measure $m_F$ is finite, it is automatically ergodic and conservative.
However, of course, the converse is not true, and it is precisely the purpose of this paper to propose criteria of finiteness.


\subsection{Some exercises in hyperbolic geometry}

This section gathers some well-known lemmas about the geometry of manifolds with pinched negative sectional curvature.

We start by recalling a very classical comparison lemma, as stated for example in \cite[Lemma 2.1]{PaP}, from which we will derive the next lemmas.

\begin{lemm}
 \label{lem:triangleestimate}
 Let $(X, d)$ be a $\text{CAT}(-1)$-space.
 For all points $x, y$ in $X$ and $z$ in $X \cup \partial_\infty X$, and every $t \in \itv[cc]{0}{d(x, z)}$ (finite if $z \in \partial_\infty X$), if $x_t$ is the point on $\itv[cc]{x}{z}$ at distance $t$ from $x$, then
 \[ d(x_t, \itv[cc]{y}{z}) \le e^{-t} \sinh(d(x, y)) \eqpunct{.} \]
\end{lemm}

In particular, applying this lemma twice leads to the following lemma.

\begin{lemm}
 \label{lem:parallelestimate}
 Let $r, r' > 0$.
 Let $x, x', y, y' \in \wt{M}$ such that $d(x, y) \le r$ and $d(x', y') \le r'$.
 For every $t \in \itv[cc]{0}{d(x, x')}$, denote by $x_t$ the point on $\itv[cc]{x}{x'}$ at distance $t$ from $x$.
 Then
 \[ d(x_t, \itv[cc]{y}{y'}) \le \sinh(r) e^{-t} + \sinh(r') e^{t-d(x, x')+\sinh(r)} \eqpunct{.} \]
\end{lemm}

Recall also the following.

\begin{lemm}[\cite{PPS}, remark 2 following lemma 3.2]
 \label{lem:potcontrol}
 For every $R \ge 0$, there is a constant $C$ that only depends on $\wt{F}$, $R$ and the bounds on the sectional curvature of $\wt{M}$ such that for every $x, x', y, y'$ in $\wt{M}$ satisfying $d(x, x'), d(y, y') \le R$
 \[ \left| \int_x^y \wt{F} - \int_{x'}^{y'} \wt{F} \right| \le C \eqpunct{.} \]
\end{lemm}


\subsection{Parallel geodesic segments avoiding images of a compact set}

The following geometrical lemma is the key ingredient of the proof of Theorem \ref{thm:crit1}, from which Theorem \ref{thm:crit3} is derived.
It asserts that, if a geodesic segment $\itv[cc]{y}{y'}$ is known to avoid the $\Gamma$-orbit of balls $B(x, R)$ except maybe at its beginning or at its end, then every other geodesic segment whose endpoints are close from $y$ and $y'$ will also essentially avoid the $\Gamma$-orbit of $\varepsilon$-shrinked balls $B(x, R - \varepsilon)$, provided that the guiding segment $\itv[cc]{y}{y'}$ is long enough.

\begin{lemm}[Long range subset avoidance]
 \label{lem:lrsa}
 Let $\wt{W}$ be an open relatively compact subset of $\wt{M}$.
 For every $R \ge 0$, there exists $\rho = \rho(R)$ such that for all $\varepsilon > 0$, all $\wt{W}' \subset \wt{W}$ open relatively compact subsets with $d(\wt{W}', \wt{M} \setminus \wt{W}) \ge \varepsilon$, all $y, y', z, z' \in \wt{M}$ with $d(y, z), d(y', z') \le R$, and all $\gamma \in \Gamma$, if
 \[ \itv[cc]{y}{y'} \cap \gamma \wt{W} = \emptyset \eqand \itv[cc]{z}{z'} \cap \gamma \wt{W}' \neq \emptyset \eqpunct{.} \]
 then we have
 \[ \min \( d(y, \gamma \wt{W}), d(y', \gamma \wt{W}) \) \le \rho - \log \varepsilon \eqpunct{.} \]
\end{lemm}

\begin{proof}
 Denote by $l = d(z, z')$.
 Let $z_t$ be a point of $\itv[cc]{z}{z'}$ inside $\gamma \wt{W}'$, with $t = d(z, z_t)$.
 Without loss of generality, we can assume that $t \le \frac{l}{2}$.
 By Lemma \ref{lem:parallelestimate}, we have
 \[ d(z_t, \itv[cc]{y}{y'}) \le \sinh(R) \( e^{-t} + e^{t - l + \sinh(R)} \) \le C \( e^{-t} + e^{t - l} \) \le 2 C e^{-t} \]
 with $C = \sinh(R) e^{\sinh(R)}$.
 The assumption $d(z_t, \itv[cc]{y}{y'}) \ge \varepsilon$ ensures that $t \le \log(\frac{2 C}{\varepsilon})$ and henceforth that $d(y, z_t) \le \rho - \log \varepsilon$ with $\rho = \log(2 C) + \diam \wt{W}$.
\end{proof}


\subsection{About finding hyperbolic isometries}

In the proofs of Theorems \ref{thm:erg} and \ref{thm:crit3}, we will need to compare sums indexed on periodic orbits of the geodesic flows, i.e. on conjugacy classes of hyperbolic elements of $\Gamma$, with sums indexed on the whole group $\Gamma$.
To this end, we need some technical tools to go from the former to the latter.
We start by recalling a variant of Anosov closing lemma, which is easily obtained by combining \cite[cor. 8.22]{GdlH} with Lemma \ref{lem:parallelestimate}.

\begin{lemm}
 \label{lem:hyper}
 For every $l, \varepsilon > 0$, there exists $\varepsilon' \in \itv[oc]{0}{1}$ with $\lim_{\varepsilon \to 0} \varepsilon' = 0$ such that for every isometry $\gamma$ of any proper geodesic $\text{CAT}(-1)$-space $X$, for every $x_0$ in $X$, if $d(x_0, \gamma x_0) \ge l$ and $d(x_0, \itv[cc]{\gamma^{-1} x_0}{\gamma x_0}) \le \varepsilon'$, then $\gamma$ is hyperbolic and $d(x_0, A_\gamma) \le \varepsilon$, where $A_\gamma$ is the translation axis of $\gamma$ in $X$.
\end{lemm}

Let $\Gamma_h$ be the set of hyperbolic elements of $\Gamma$.
If $\gamma \in \Gamma_h$, we denote by $A_\gamma \subset T^1 \wt{M}$ its \name{axis}, i.e. the set of vectors $v \in T^1 \wt{M}$ such that $g^{l(\gamma)}v = \gamma v$, where $l(\gamma)$ is the minimal displacement of a point by $\gamma$.
In other words, $A_\gamma$ is the set of unit vectors on the geodesic joining the repulsive fixed point to the attractive fixed point, oriented towards the latter.

If $x \in \wt{M}$, $l \ge 0$ and $U \subset \partial_\infty \wt{M}$ is open, then the \name{angular sector at distance $l$ based at $x$ and supported by $U$} is the open set
\[ \mathcal{C}_{x,l}(U) = \set{z \in \wt{M}}{d(z, x) > l \eqand \exists \xi \in U, z \in \itv[oo]{x}{\xi}} \eqpunct{.} \]

\begin{lemm}
 \label{lem:forcehyper}
 Let $\wt{\W} \subset T^1 \wt{M}$ be an open relatively compact set intersecting $\wt{\Omega}$, $\varepsilon > 0$ and $x \in \pi(\wt{\W}) \cap \Conv(\Lambda(\Gamma))$.
 There exist $g_1, \hdots, g_k \in \Gamma$ and a finite set $S \subset \Gamma$ such that for every $\gamma \in \Gamma \setminus S$, there exist $i, j$ such that $\gamma' = g_j^{-1} \gamma g_i$ is hyperbolic and its axis satisfies $A_{\gamma'} \cap \wt{\W} \cap T^1 B(x, \varepsilon) \neq \emptyset$.
\end{lemm}

\begin{proof}
 Take $\varepsilon' = \varepsilon'(\varepsilon, 1)$ given by Lemma \ref{lem:hyper}.
 Let $U$ and $V$ be two non empty open sets of $\partial_\infty \wt{M}$ with disjoint closures, both meeting $\Lambda(\Gamma)$, such that any geodesic orbit of $T^1 \wt{M}$ from $U$ to $V$ meets $\wt{\W} \cap T^1 B(x, \varepsilon')$.
 Fix two non empty open sets $U_0, V_0 \subset \partial_\infty \wt{M}$ meeting $\Lambda(\Gamma)$ such that $\overline{U_0} \subset U$ and $\overline{V_0} \subset V$.
 There exists $l' \ge l$ such that for every $y \in U' = \mathcal{C}_{x,l'}(U_0)$ and $z \in V' = \mathcal{C}_{x,l'}(V_0)$, the geodesic orbit of $T^1 \wt{M}$ from $y$ to $z$ meets $\wt{\W} \cap T^1 B(x, \varepsilon')$.

 As $\Gamma$ acts minimally on $\Lambda(\Gamma)$, there exist $g_1, \hdots, g_p, \hdots, g_k \in \Gamma$ such that
 \[ \Lambda(\Gamma) \subset \bigcup_{i=1}^p g_i U_0 \eqand \Lambda(\Gamma) \subset \bigcup_{i=p+1}^k g_i V_0 \eqpunct{.} \]
 Let $R_0 = \sup \set{d(x, g_i^{-1} x)}{i = 1, \hdots, k}$, $R_1 = R_0 + 2 \varepsilon$, and define
 \[ U'' = \set{y \in \wt{M}}{B(y, R_1) \subset U'} \eqand V'' = \set{z \in \wt{M}}{B(z, R_1) \subset V'} \eqpunct{.} \]
 Observe that $\partial_\infty U'' = U_0$ and $\partial_\infty V'' = V_0$, so that we still have in $\wt{M} \cup \partial_\infty \wt{M}$
 \[ \Lambda(\Gamma) \subset \bigcup_{i=1}^p g_i U'' \eqand \Lambda(\Gamma) \subset \bigcup_{i=p+1}^k g_i V'' \eqpunct{.} \]
 Therefore, both sets
 \[ K = \Conv(\Lambda(\Gamma)) \setminus \bigcup_{i=1}^p g_i U'' \eqand L = \Conv(\Lambda(\Gamma)) \setminus \bigcup_{i=p+1}^k g_i V'' \]
 are closed and do not meet $\partial_\infty \wt{M}$, hence they are compact in $\wt{M}$.
 Therefore, $S = \set{\gamma \in \Gamma}{\gamma^{-1} x \in K \eqor \gamma x \in L}$ is finite.

 Now if $\gamma \in \Gamma \setminus S$, by construction, there exist $i, j$ such that $\gamma^{-1} x \in g_i U''$ and $\gamma x \in g_j V''$.
 If $\gamma' = g_j^{-1} \gamma g_i$, then $u = {\gamma'}^{-1} g_j^{-1} x \in U''$ and $v = \gamma' g_i^{-1} x \in V''$, which means that $\gamma' x$ satisfies
 \[ d(\gamma' x, v) = d(x, g_i^{-1} x) \le R_0 < R_1 \eqpunct{,} \]
 i.e. $\gamma' x \in V'$ and similarly ${\gamma'}^{-1} x \in U'$.
 By our choice of $U'$ and $V'$, this ensures that $d(x, \gamma' x) \ge l$ and $\itv[cc]{{\gamma'}^{-1} x}{\gamma' x}$ meets $B(x, \varepsilon')$, so that by Lemma \ref{lem:hyper}, $\gamma'$ is hyperbolic and its axis $A_{\gamma'}$ meets $T^1 B(x, \varepsilon)$.

 Finally, let $z \in \wt{\pi}(A_{\gamma'}) \cap B(x, \varepsilon)$.
 Since
 \[ d(\gamma' z, v) \le d(\gamma' z, \gamma' x) + d(\gamma' x, v) \le \varepsilon + R_0 < R_1 \eqpunct{,} \]
 we deduce that $\gamma' z \in V'$ and likewise ${\gamma'}^{-1} z \in U'$.
 This implies that the geodesic orbit of $T^1 \wt{M}$ from ${\gamma'}^{-1} z$ to $\gamma' z$, i.e. $A_{\gamma'}$, also meets $\wt{\W}$.
\end{proof}


\subsection{Shadows}

If $y, y'$ are two distinct points of $\wt{M} \cup \partial_\infty \wt{M}$, let $v_-(y, y') \in \partial_\infty \wt{M}$ (respectively $v_+(y, y') \in \partial_\infty \wt{M}$) be the endpoint of the one-sided infinite geodesic ray going from $y'$ to $y$ (respectively from $y$ to $y'$).
The maps $v_+$ and $v_-$ are continuous for the usual topology on $(\wt{M} \cup \partial_\infty \wt{M})^2 \setminus \Diag$.

With the above notations, if $x \in \wt{M} \cup \partial_\infty \wt{M}$ and $W$ is an open, relatively compact, geodesically convex subset of $\wt{M}$, the \name{shadow of $W$ viewed from $x$} is the set
\[ \O_x W = \set{v_+(x, y)}{y \in W} = \set{\xi \in \partial_\infty \wt{M}}{\itv[oo]{x}{\xi} \cap W \neq \emptyset} \eqpunct{.} \]

We start by stating a classical lemma that asserts that, if the base point is far enough from the set that casts the shadow, then it can be moved around by a bounded amount almost without changing the shadow.

\begin{lemm}
 \label{lem:shadowcontrol}
 For every $r > 0$, every $0 < \varepsilon < r$ and every $\delta > 0$, there exists $l_0 \ge 0$ such that for all $x, y \in \wt{M}$ satisfying $d(x, y) \ge l_0$ and for all $z \in B(y, \delta)$ we have
 \[ \O_y B(x, r - \varepsilon) \subset \O_z B(x, r) \subset \O_y B(x, r + \varepsilon) \eqpunct{.} \]
\end{lemm}

We will also need the two following lemmas about products of shadows.
Their proofs are very similar to \cite[Lemma 3.17]{PPS} and therefore ommitted.

\begin{lemm}
 \label{lem:shadowup}
 For all $r > 0$, $r' > 0$ and $\varepsilon > 0$, there exists $l_0 \ge 0$ such that for all $x, x' \in \wt{M}$ satisfying $d(x, x') \ge l_0$, for all $y \in B(x, r)$ and $y' \in B(x', r')$, we have
 \[ (v_-(y, y'), v_+(y, y')) \in \O_{x'} B(x, r + \varepsilon) \times \O_x B(x', r' + \varepsilon) \eqpunct{.} \]
\end{lemm}

\begin{lemm}
 \label{lem:shadowdown}
 For every $r > 0$, $r' > 0$, and $\varepsilon > 0$, there exists $l_0 \ge 0$ such that the following holds : for every $x, x' \in \wt{M}$ satisfying $d(x, x') \ge l_0$, for every $v_- \in \O_{x'} B(x, r)$ and every $v_+ \in \O_x B(x', r')$, there exist $y \in B(x, r + \varepsilon)$ and $y' \in B(x', r' + \varepsilon)$ such that
 \[ (v_-, v_+) = (v_-(y, y'), v_+(y, y')) \eqpunct{.} \]
\end{lemm}

The next lemma states that when two balls are far enough one from each other, their shadows relative to each other's center cannot intersect.

\begin{lemm}
 \label{lem:shadowdisjoint}
 For every $R \ge 0$, there exists a $l_0 \ge 0$ such that, for every $x, y \in \wt{M}$ satisfying $d(x, y) \ge l_0$, one has
 \[ \O_x B(y, R) \cap \O_y B(x, R) = \emptyset \eqpunct{.} \]
\end{lemm}

\begin{proof}
 Suppose not, and take $\xi$ in the intersection.
 In the triangle $(x, y, \xi)$, we would have $d(x, \itv[co]{y}{\xi}) \le R$ and $d(y, \itv[co]{x}{\xi}) \le R$.
 Triangles of $\wt{M}$ are $\delta$-hyperbolic for some positive constant $\delta$ depending only on the upper bound of the sectional curvature.
 Therefore, this situation is possible only if $d(x, y) \le 2 R + 2 \delta$.
 Thus, the lemma is proved with $l_0 = 2 R + 2 \delta$.
\end{proof}

\subsubsection*{Measure of shadows}

The Shadow Lemma, initially due to Sullivan, estimates the measure given by Patterson-Sullivan-Gibbs densities to shadows of balls in terms of integrals of the normalized potential.
It has been proven by Mohsen in our setting, and asserts the following.

\begin{lemm}[Mohsen's Shadow Lemma, Lemma 3.10 in \cite{PPS}]
 \label{lem:shadowmohsen}
 Let $(\mu_x^F)_{x \in \wt{M}}$ be the Patterson-Sullivan-Gibbs conformal density associated with $F$, and $K$ be a compact subset of $\wt{M}$.
 There exists $R_0 > 0$ such that, for all $R \ge R_0$, there exists $C > 0$ such that for all $\gamma \in \Gamma$ and $x, y \in K$
 \[ \frac{1}{C} e^{\int_x^{\gamma y} (\wt{F} - P(F))} \le \mu_x^F\( \O_x B(\gamma y, R) \) \le C e^{\int_x^{\gamma y} (\wt{F} - P(F))} \eqpunct{.} \]
\end{lemm}

A careful examination of the proof of this lemma shows that the condition $R \ge R_0$ is actually only necessary for the lower bound.
In the following, we will only use this lemma for its upper bound, so we can forget about this restriction.

However, we will also need some lower bound estimates for the $m_F$-measure of dynamical balls.
To this end, we will use the following variant of the Shadow Lemma for product of shadows of balls, which replaces the restriction on the size of balls by the assumption that the ball intersects the nonwandering set of the geodesic flow.

\begin{lemm}[Shadow product lemma]
 \label{lem:shadowprod}
 Let $(\mu_x^F)_{x \in \wt{M}}$ and $(\mu_x^{F \circ \iota})_{x \in \wt{M}}$ be the Patterson-Sullivan-Gibbs conformal densities respectively associated with $F$ and $F \circ \iota$.
 Assume that $B(x, R) \subset \wt{M}$ intersects the base of the nonwandering set $\wt{\Omega}$.
 Then there exist $C > 0$ and $S, G \subset \Gamma$ finite such that for every $\gamma \in \Gamma \setminus S$ there exist $g, h \in G$ such that
 \[ \frac{1}{C} e^{\int_x^{\gamma x} (\wt{F} - P(F))} \le \mu_x^{F \circ \iota}(\O_{\gamma x} B(g x, R)) \mu_x^F(\O_x B(\gamma h x, R)) \le C e^{\int_x^{\gamma x} (\wt{F} - P(F))} \eqpunct{.} \]
\end{lemm}

\begin{proof}
 If $T^1 B(x, R) \cap \wt{\Omega} \neq \emptyset$, then we can find $\eta, \xi \in \Lambda(\Gamma)$ distinct such that the geodesic $(\xi \eta)$ intersects $B(x, R)$.
 In particular there exists $\varepsilon > 0$ such that $\eta \in \O_\xi B(x, R - \varepsilon)$ and $\xi \in \O_\eta B(x, R - \varepsilon)$.
 By continuity of the shadows, there exist two neighbourhoods $U, V$ of respectively $\xi$ and $\eta$ in $\wt{M} \cup \partial_\infty \wt{M}$ such that
 \[ \forall y \in U, \O_\xi B(x, R - \varepsilon) \subset \O_y B(x, R) \eqand \forall z \in V, \O_\eta B(x, R - \varepsilon) \subset \O_z B(x, R) \eqpunct{.} \]
 By using the same technique as in the proof of lemma \ref{lem:forcehyper}, we can find $S, G \subset \Gamma$ finite such that, for every $\gamma \in \Gamma \setminus S$, there exist $g, h \in G$ such that $g^{-1} \gamma x \in U$ and $h^{-1} \gamma^{-1} x \in V$.

 Since the Patterson-Sullivan-Gibbs densities charge any open set that intersects the limit set,
 \[ \alpha = \min \set{\mu_{g^{-1} x}^{F \circ \iota}(\O_\xi B(x, R - \varepsilon))}{g \in G} > 0 \eqand \beta = \mu_x^F(\O_\eta B(x, R - \varepsilon)) > 0 \eqpunct{.} \]
 Let $\gamma \in \Gamma \setminus S$ and take $g, h \in G$ such that $g^{-1} \gamma x \in U$ and $h^{-1} \gamma^{-1} x \in V$.
 By the invariance property of the densities, we have that on the one hand that
 \[ \mu_x^{F \circ \iota}(\O_{\gamma x} B(g x, R)) = \mu_{g^{-1} x}^{F \circ \iota}(\O_{g^{-1} \gamma x} B(x, R)) \ge \mu_{g^{-1} x}^{F \circ \iota}(\O_\xi B(x, R - \varepsilon)) \ge \alpha \eqpunct{,} \]
 and on the other hand that
 \[ \mu_{\gamma h x}^F(\O_x B(\gamma h x, R)) = \mu_x^F(\O_{h^{-1} \gamma^{-1} x} B(x, R)) \ge \beta \eqpunct{.} \]
 But the conformal density property of $(\mu_x^F)$ ensures that
 \[ \mu_x^F(\O_x B(\gamma h x, R)) = \int_{\zeta \in \O_x B(\gamma h x, R)} e^{-C_{F-P(F),\zeta}(x, \gamma h x)} d\mu_{\gamma h x}^F(\zeta) \eqpunct{,} \]
 where $C_{F,\zeta}(x, y) = \lim_{t \to \infty} \int_y^{\zeta_t} \wt{F} - \int_x^{\zeta_t} \wt{F}$ is the Gibbs cocycle associated with $F$.
 By applying \cite[Lemma 3.4]{PPS} $(2)$, we get the existence of a constant $C_1 \ge 1$ independent of $\gamma$ and $h$ such that
 \[ \frac{1}{C_1} e^{\int_x^{\gamma h x} (\wt{F} - P(F))} \le \frac{\mu_x^F(\O_x B(\gamma h x, R))}{\mu_{\gamma h x}^F(\O_x B(\gamma h x, R))} \le C_1 e^{\int_x^{\gamma h x} (\wt{F} - P(F))} \eqpunct{.} \]
 This implies that
 \[ \frac{\alpha \beta}{C_1} e^{\int_x^{\gamma h x} (\wt{F} - P(F))} \le \mu_x^{F \circ \iota}(\O_{\gamma h x} B(x, R)) \mu_x^F(\O_x B(\gamma h x, R)) \le C_1 e^{\int_x^{\gamma h x} (\wt{F} - P(F))} \eqpunct{.} \]
 Finally, after noting that $d(\gamma x, \gamma h x) = d(x, h x)$ is bounded from above independently from $\gamma$, we apply Lemma \ref{lem:potcontrol} to obtain a constant $C_2 \ge 0$ that only depends on $\wt{M}$, $\wt{F}$, $x$ and $R$ such that
 \[ \left| \int_x^{\gamma h x} (\wt{F} - P(F)) - \int_x^{\gamma x} (\wt{F} - P(F)) \right| \le C_2 \eqpunct{.} \]
 This concludes the proof with $C = \frac{C_1}{\alpha \beta} e^{C_2}$.
\end{proof}


\section{Number of returns of a periodic orbit}

\label{sec:numret}

The aim of this section is to introduce an useful mathematical definition of the "number of times that a periodic geodesic enters in a given set $\W$".
Observe that as soon as $\W$ is non convex, or has holes, it may be highly non trivial if done in a too naive way.

Let $\wt{\W}$ be a relatively compact subset of $T^1 \wt{M}$.
If $\gamma \in \Gamma_h$, we define the \name{number of copies of the axis of $\gamma$ intersecting $\wt{\W}$} as the quantity
\[ n_{\wt{\W}}(\gamma) = \card \set{\gamma' \in \Gamma}{\exists g \in \Gamma, \gamma' = g^{-1} \gamma g \eqand A_{\gamma'} \cap \wt{\W} \neq \emptyset} \eqpunct{.} \]
By definition, this number depends only on the conjugacy class of $\gamma \in \Gamma_h$.
Of course, it is also $\Gamma$-invariant, in the sense that
\[ n_{\wt{\W}}(\gamma) = n_{g \wt{\W}}(\gamma) \eqpunct{.} \]

We shall now extend this definition to relatively compact subsets of $T^1 M$ in the following way.
First note that if $\W \subset T^1 M$ is open and relatively compact, then it admits an open relatively compact lift (actually many of them), i.e. an open relatively compact set $\wt{\W} \subset T^1 \wt{M}$ such that $\Pr(\wt{\W}) = \W$ where $\fn{\Pr}{T^1 \wt{M}}{T^1 M}$ is the covering map.
Indeed, it is enough to cover $\W$ by trivializing open sets for the covering $\Pr$, to take for each of these sets the image of its intersection with $\W$ by one of the inverse branches of $\Pr$, and then let $\wt{\W}$ to be the union of these preimages.
However, there might not be an open lift $\wt{\W}$ of $\W$ such that $\fn{\Pr}{\wt{\W}}{\W}$ is $1-1$ if, for example, the base $\pi(\W)$ contains a ball $B(x, R)$ whose radius $R$ is larger than the injectivity radius at $x$.

Given $\W \subset T^1 M$ open relatively compact, and any periodic orbit $p \in \P$, this leads us to define the \name{number of returns of $p$ into $\W$} by
\[ n_\W(p) = \inf n_{\wt{\W}}(\gamma_p) \eqpunct{,} \]
where $\gamma_p$ is any hyperbolic isometry in the conjugcy class associated with $p$, and the infimum is taken over all open relatively compact lifts $\wt{\W}$ of $\W$ to $T^1 \wt{M}$.

Note that the quantities $n_{\wt{\W}}(\gamma)$ and $n_\W(p)$ do not depend on the multiplicity of $\gamma$ or $p$.
Indeed, two isometries $\gamma$ and $\gamma'$ are conjugated by an element $g$ if and only if $\gamma^k$ and ${\gamma'}^k$ are conjugated by this element $g$, for $k \ge 1$.
Moreover, the axii $A_{\gamma}$ and $A_{{\gamma}^k}$ are equal for all $k \ge 1$.
Therefore, we get
\[ \forall k \ge 1, n_{\wt{\W}}(\gamma^k) = n_{\wt{\W}}(\gamma) \eqpunct{.} \]
Equivalently, if $p \in \P$ is a periodic orbit with multiplicity whose associated primitive orbit is $p_0 \in \P'$, we have
\[ n_\W(p) = n_\W(p_0) \eqpunct{.} \]

Although there might not exist a lift of $\W$ that realizes the number of returns $n_{\W}(p)$ of $p \in \P$ into $\W$ as a number of copies, the number of copies of the axis of $\gamma \in \Gamma_h$ intersecting two lifts of $\W$ are uniformly commensurable with each other.

\begin{lemm}
 \label{lem:crossingcompare}
 Let $\wt{\W}_1, \wt{\W}_2$ be two lifts of $\W$ to $T^1 \wt{M}$.
 Then there is a $C = C_{\wt{\W}_1, \wt{\W}_2}$ such that
 \[ \forall \gamma \in \Gamma_h, \frac{1}{C} n_{\wt{\W}_2}(\gamma) \le n_{\wt{\W}_1}(\gamma) \le C n_{\wt{\W}_2}(\gamma) \eqpunct{.} \]
\end{lemm}

\begin{proof}
 It is enough to show that $n_{\wt{\W}_1}(\gamma_0) \le C n_{\wt{\W}_2}(\gamma_0)$ with $\gamma_0 \in \Gamma_h'$.
 Take $\gamma' = g^{-1} \gamma_0 g$ for some $g \in \Gamma$ such that there is a $v \in A_{\gamma'} \cap \wt{\W}_1$.
 Since $\wt{\W}_1$ and $\wt{\W}_2$ are both lifts of the same set $\W$, the set
 \[ H_v = \set{h \in \Gamma}{h v \in \wt{\W}_2} \]
 is non empty.
 Note that if $h \in H_v$, then $h(v) \in \wt{\W}_2 \cap A_{\gamma''}$ where $\gamma'' = h \gamma' h^{-1}$ is a conjugate of $\gamma_0$.
 This ensures that
 \begin{align*}
  &\set{\gamma' \in \Gamma}{\exists g \in \Gamma, \gamma' = g^{-1} \gamma_0 g \eqand A_{\gamma'} \cap \wt{\W}_1 \neq \emptyset} \\
  &\hspace{2em} \subset \bigcup_{h \in H} h^{-1} \set{\gamma'' \in \Gamma}{\exists g \in \Gamma, \gamma'' = g^{-1} \gamma_0 g \eqand A_{\gamma''} \cap \wt{\W}_2 \neq \emptyset} h \eqpunct{,}
 \end{align*}
 where $H = \cup_{v \in \wt{\W}_1} H_v$ depends only on $\wt{\W}_1$ and $\wt{\W}_2$ but not on $\gamma_0$.
 In order to conclude, it is enough to show that $H$ is finite.
 Indeed, let $\wt{\W}_3 = \overline{\wt{\W}_1 \cup \wt{\W}_2}$.
 It is a compact subset of $T^1 \wt{M}$ and we have
 \[ H = \set{h \in \Gamma}{\exists v \in \wt{\W}_1, h v \in \wt{\W}_2} \subset \set{h \in \Gamma}{\wt{\W}_3 \cap h \wt{\W}_3 \neq \emptyset} \eqpunct{,} \]
 which is finite since the action of $\Gamma$ on $T^1 \wt{M}$ is proper.
\end{proof}

In particular, if one takes a relatively compact lift $\wt{\W}_1$ of $\W$, then there exists a constant $C$ which depends only on $\wt{\W}_1$ such that for every $p_0 \in \P'$ and every $\gamma_0$ in the conjugacy class associated with $p_0$ we have
\[ \frac{1}{C} n_{\wt{\W}_1}(\gamma_0) \le n_\W(p_0) \le C n_{\wt{\W}_1}(\gamma_0) \eqpunct{.} \]

\begin{lemm}
 \label{lem:crossingcover}
 If $\W \subset T^1 M$ is covered by a finite collection $(\W_i)_{i = 1, \hdots, n}$ of open relatively compact subsets of $T^1 M$, then there exists $C = C_{\W, \W_1, \hdots, \W_n} > 0$ such that
 \[ \forall p \in \P, n_\W(p) \le C \sum_{i=1}^n n_{\W_i}(p) \eqpunct{.} \]
\end{lemm}

\begin{proof}
 We may assume that $p \in \P'$.
 Fix $\gamma_p \in \Gamma_h'$ in the conjugacy class associated with $p$.
 For each $i$, take an open relatively compact lift $\wt{\W}_i$ of $\W_i$ to $T^1 \wt{M}$, as well as a constant $C_i$ independent from $p$ and $\gamma_p$ such that
 \[ n_{\wt{W}_i}(\gamma_p) \le C_i n_{\W_i}(p) \eqpunct{.} \]
 Observe that
 \[ \wt{\W} = \Pr^{-1}(\W) \cap \bigcup_{i=1}^n \wt{\W}_i \]
 is an open relatively compact lift of $\W$, and that if an axis $A_\gamma$ meets $\wt{\W}$ then it meets at least one of the $\wt{\W}_i$.
 Therefore
 \begin{equation*}
  n_\W(p) \le n_{\wt{\W}}(\gamma_p) \le \sum_{i=1}^n n_{\wt{\W}_i}(\gamma_p) \le \max(C_i) \sum_{i=1}^n n_{\W_i}(p) \eqpunct{.} \qedhere
 \end{equation*}
\end{proof}


\section{Ergodicity of Gibbs measures for recurrent potentials}

\label{sec:erg}

Let $\fn{F}{T^1 M}{\R}$ be a H\"older continuous potential on $T^1 M$.
According to the Hopf-Tsuji-Sullivan Theorem, we know that $m_F$ is ergodic and conservative if and only if the Poincar\'e series associated with $F$ diverges at the critical exponent $s = P(F)$, in which case $(\Gamma, F)$ is said to be \name{divergent} following the terminology in \cite{PPS}.
In this section, we will prove Theorem \ref{thm:erg} which asserts that it is also equivalent to the divergence of the series
\[ \sum_{p \in \P} n_\W(p) e^{\int_p(F - P(F))} \eqpunct{,} \]
for $\W$ an open relatively compact set intersecting $\Omega$.
$F$ is said to be \name{recurrent} relatively to $\W$ when this series diverges.

Note that periodic orbits meeting $\W$ are the only periodic orbits to consider in the above sum, because otherwise $n_\W(p) = 0$.

\begin{theo.recite1}
 Let $M$ be a negatively curved orbifold with pinched negative curvature, and $\fn{F}{T^1 M}{\R}$ a H\"older continuous potential with $P(F) < +\infty$.
 Then the Gibbs measure $m_F$ is ergodic and conservative if and only if $F$ is recurrent with respect to some open relatively compact set intersecting $\Omega$.
\end{theo.recite1}

In particular, the recurrence property does not depend on the choice of the open relatively compact subset $\W$.

Observe first that, for any real number $k$, this equivalence is satisfied for a potential $F$ if and only if it is satisfied for the potential $F + k$, as the Gibbs measures $m_F$ and $m_{F+k}$ are equal.
We may therefore assume from now on that $P(F) = 0$.
We will also denote by $\wt{F}$ the $\Gamma$-invariant lift of $F$ to $T^1 \wt{M}$.


\subsection{Recurrence implies divergence}

\begin{lemm}
 If $F$ is recurrent relatively to some open relatively compact subset $\W$ of $T^1 M$ which intersects $\Omega$, then $(\Gamma, F)$ is divergent.
\end{lemm}

\begin{proof}
 Let $\wt{\W}$ be an open relatively compact lift of $\W$ to $T^1 \wt{M}$ that meets $\wt{\Omega}$.
 Choose a base point $x \in \pi(\wt{\W}) \subset \wt{M}$.
 If $p \in \P$ intersects $\W$, let $\gamma_{p,1}, \hdots, \gamma_{p,n_p} \in \Gamma_h$ be the distinct hyperbolic isometries whose axii intersect $\wt{\W}$ and project onto the periodic orbit $p$.
 According to Lemma \ref{lem:crossingcompare}, there exists $C > 0$ such that for every $1 \le i \le n_p$ one has
 \[ n_p = n_{\wt{\W}}(\gamma_{p,i}) \ge \frac{1}{C} n_\W(p) \eqpunct{.} \]
 For each $1 \le i \le n_p$, pick $z_i(p) \in A_{\gamma_{p,i}} \cap \wt{\W}$ and let $x_i(p) = \pi(z_i(p)) \in \wt{M}$.
 In particular, we have $d(x, x_i(p)) \le \diam \pi(\wt{\W})$.

 According to Lemma \ref{lem:potcontrol}, there is a constant $C' \ge 0$ that only depends on $F$, on $\diam \pi(\wt{\W})$ and on the bounds on the sectional curvature of $\wt{M}$ such that
 \[ \forall p \in \P, \forall i, \left| \int_p F - \int_x^{\gamma_{p,i} x} \wt{F} \right| = \left| \int_{x_i(p)}^{\gamma_{p,i} x_i(p)} \wt{F} - \int_x^{\gamma_{p,i} x} \wt{F} \right| \le C' \eqpunct{.} \]
 Hence we have
 \[ \sum_{p \in \P} n_\W(p) e^{\int_p F} \le C e^{C'} \sum_{p \in \P} \sum_{i=1}^{n_p} e^{\int_x^{\gamma_{p,i} x} \wt{F}} \le C e^{C'} \sum_{\gamma \in \Gamma} e^{\int_x^{\gamma x} \wt{F}} \eqpunct{,} \]
 and the recurrence of $F$ implies the divergence of $(\Gamma, F)$.
\end{proof}


\subsection{Divergence implies recurrence}

\begin{lemm}
 If $(\Gamma, F)$ is divergent, then $F$ is recurrent relatively to any open relatively compact subset $\W$ of $T^1 M$ which intersects $\Omega$.
\end{lemm}

\begin{proof}
 Choose an open relatively compact lift $\wt{\W}$ of $\W$ to $T^1 \wt{M}$.
 Assume that $(\Gamma, F)$ is divergent.
 As divergence is independent from the chosen base point $x \in \wt{M}$, consider $x \in \wt{\Omega} \cap \pi(\wt{\W})$.

 According to Lemma \ref{lem:forcehyper}, there exist $g_1, \hdots, g_k \in \Gamma$ and $S \subset \Gamma$ finite such that
 \[ \sum_{\gamma \in \Gamma} e^{\int_x^{\gamma x} \wt{F}} \le \sum_{\gamma \in S} e^{\int_x^{\gamma x} \wt{F}} + \sum_{i,j} \sum_{\substack{\gamma' \in \Gamma_h \\ A_{\gamma'} \cap \wt{\W} \neq \emptyset}} e^{\int_x^{g_j \gamma' g_i^{-1} x} \wt{F}} \eqpunct{.} \]
 Let $R_0 = \sup\set{d(x, g_i x)}{i = 1, \hdots, k}$.
 Lemma \ref{lem:potcontrol} ensures that there exists a constant $C_1 \ge 0$ which depends only on $F$, $R_0$ and the bounds on the sectional curvature of $\wt{M}$ such that for all $1 \le i, j \le k$,
 \[ \left| \int_x^{g_i \gamma' g_j^{-1} x} \wt{F} - \int_x^{\gamma' x} \wt{F} \right| = \left| \int_{g_i^{-1} x}^{\gamma' g_j^{-1} x} \wt{F} - \int_x^{\gamma' x} \wt{F} \right| \le C_1 \eqpunct{.} \]
 From this we deduce that
 \[ \sum_{\gamma \in \Gamma} e^{\int_x^{\gamma x} \wt{F}} \le \sum_{\gamma \in S} e^{\int_x^{\gamma x} \wt{F}} + k^2 e^{C_1} \sum_{\substack{\gamma' \in \Gamma_h \\ A_{\gamma'} \cap \wt{\W} \neq \emptyset}} e^{\int_x^{\gamma' x} \wt{F}} \eqpunct{.} \]
 We can now reindex this last sum by summing first over $p \in \P$, and then over the hyperbolic isometries $\gamma$ associated with $p$, i.e. those such that $\Pr(A_\gamma) = p$, as follows.
 \[ \sum_{\substack{\gamma \in \Gamma_h \\ A_\gamma \cap \wt{W} \neq \emptyset}} e^{\int_x^{\gamma x} \wt{F}} = \sum_{p \in \P} \sum_{\substack{\gamma \in \Gamma_h \\ A_\gamma \cap \wt{\W} \neq \emptyset \\ \Pr(A_\gamma) = p}} e^{\int_x^{\gamma x} \wt{F}} \eqpunct{.} \]
 Recall that for all $p \in \P$ meeting $\W$, there exist only finitely many $\gamma_{p,1}, \hdots, \gamma_{p,n_p} \in \Gamma_h$ that project onto $p$ and whose axis intersects $\wt{\W}$.
 Thus this last sum is equal to
 \[ \sum_{p \in \P} \sum_{i=1}^{n_p} e^{\int_x^{\gamma_{p,i} x} \wt{F}} \eqpunct{.} \]
 But according to Lemma \ref{lem:crossingcompare}, there exists $C_2 > 0$ such that for every $1 \le i \le n_p$ one has
 \[ n_p = n_{\wt{\W}}(\gamma_{p,i}) \le C_2 n_\W(p) \eqpunct{.} \]
 As before, pick for every $1 \le i \le n_p$ a point $x_i(p) \in \wt{\pi}(A_{\gamma_{p,i}} \cap \wt{\W})$, that satisfies $d(x_i(p), x) \le \diam \pi(\wt{\W})$.
 Lemma \ref{lem:potcontrol} gives a constant $C_3 \ge 0$ which depends only on $\wt{F}$, on $\diam \pi(\wt{\W})$ and on the geometry of $\wt{M}$ such that for all $p \in \P$ and $1 \le i \le n_p$,
 \[ \left| \int_x^{\gamma_{p,i} x} \wt{F} - \int_p F \right| = \left| \int_x^{\gamma_{p,i} x} \wt{F} - \int_{x_i(p)}^{\gamma_{p,i} x_i(p)} \wt{F} \right| \le C_3 \eqpunct{.} \]
 We finally get that
 \[ \sum_{\substack{\gamma \in \Gamma_h \\ A_\gamma \cap \wt{\W} \neq \emptyset}} e^{\int_x^{\gamma x} \wt{F}} \le C_2 e^{C_3} \sum_{p \in \P} n_\W(p) e^{\int_p F} \eqpunct{,} \]
 and the divergence of $(\Gamma, F)$ implies the recurrence of $F$ relatively to $\W$.
\end{proof}


\section{Finiteness of Gibbs measures for positive recurrent potentials}

\label{sec:fin1}

The aim of this section is to prove Theorem \ref{thm:crit1}.
As noted the introduction, the only efficient tool to see whether a measure is finite or not is the so-called Kac Lemma.
Unfortunately, it is not very easy to use, and also usually stated for single transformations, and not for flows.
Translating this statement in the criterion of Theorem \ref{thm:crit1} is the work done below.


\subsection{Positive recurrence}

For $\wt{W}$ an open relatively compact subset of $\wt{M}$, we define the set
\[ \Gamma_{\wt{W}} = \set{\gamma \in \Gamma}{\exists y, y' \in \wt{W}, \itv[cc]{y}{\gamma y'} \cap g \wt{W} \neq \emptyset \Rightarrow \overline{\wt{W}} \cap g \overline{\wt{W}} \neq \emptyset \eqor \gamma \overline{\wt{W}} \cap g \overline{\wt{W}} \neq \emptyset} \]
of elements $\gamma$ such that some interval $\itv[cc]{y}{\gamma y'}$ intersects $\Gamma \wt{W}$ only around $y$ or $\gamma y'$.
When $\wt{W} = B(x, R)$ for some $x \in \wt{M}$ and $R > 0$, we shorten this notation into $\Gamma_{x,R} = \Gamma_{B(x, R)}$.
Note that for all $g \in \Gamma$, $\Gamma_{g \wt{W}} = g \Gamma_{\wt{W}} g^{-1}$.

For any open and relatively compact set $\wt{\mathcal{W}} \subset T^1 \wt{M}$, and any $\gamma \in \Gamma$, we define the \name{$\wt{\mathcal{W}},\gamma$-geodesic ball} by
\[ \mathcal{U}_{\wt{\mathcal{W}},\gamma} = \set{v \in \wt{\mathcal{W}}}{\exists t \ge 0, \wt{g}^t(v) \in \gamma \wt{\mathcal{W}}} \eqpunct{.} \]
When $\wt{\mathcal{W}} = T^1 B(x, R)$ for some $x \in \wt{M}$ and $R > 0$, we simplify this notation into
\[ \mathcal{U}_{x,R,\gamma} = \mathcal{U}_{T^1 B(x,r),\gamma} \eqpunct{.} \]

We recall Definition \ref{def:pos-rec-bis} for the reader's convenience.
The pair $(\Gamma, \wt{F})$ is said to be \name{positive recurrent relatively to $\wt{W}$} for some open relatively compact set $\wt{W}$ if $P(F)$ is finite, $(\Gamma, F)$ is divergent and the series
\[ \sum_{\gamma \in \Gamma_{\wt{W}}} d(z, \gamma z) e^{\int_z^{\gamma z} (\wt{F} - P(F))} < +\infty \]
for some $z \in \wt{M}$.
According to Lemma \ref{lem:potcontrol}, the behaviour of this series does not depend on the choice of the point $z \in \wt{M}$.
By replacing $F$ by $F - P(F)$, which does not change $m_F$, we may also assume that $P(F) = 0$.


\subsection{A Kac lemma for flows}

Let $(X, \mathcal{B}, \mu)$ be a measured space.
If $\fn{f}{X}{X}$ and $Y \in \mathcal{B}$, the \name{return time map $T_Y$ to $Y$} is defined by
\[ \func{T_Y}{Y}{\N \cup \sing{+\infty}}{y}{\inf \set{n \ge 1}{f^n(y) \in Y}} \eqpunct{.} \]
We start by recalling the statement of Kac's recurrence lemma for measure-preserving invertible transformations, as stated in \cite{Aar} or \cite{Wri}.
Note that this lemma usually requires that the transformation is conservative and ergodic, but the classical partitioning proof shows that these hypothesis can be replaced by assuming that the return time is finite almost everywhere in $Y$ and by looking only at points that can be reached from $Y$.

\begin{lemm}[Kac's lemma for transformations]
 \label{lem:kaccompute}
 Let $(X, \mathcal{B}, f, \mu)$ be a measured dynamical system, with $f$ invertible.
 If the return time map $T_Y$ is finite $\mu$-almost everywhere on $Y$, then
 \[ \sum_{n \ge 1} n \mu\( \left\{ T_Y = n \right\} \) = \mu\( \bigcup_{n \in \Z} f^n(Y) \) \eqpunct{.} \]
\end{lemm}

Let $X$ be a locally compact topological space equipped with the Borel $\sigma$-algebra, and let $(g^t)$ be a continuous flow on $X$.
For every $W \subset X$ open and $\varepsilon > 0$, we define the \name{$\varepsilon$-hitting time of $W$} by
\[ \forall x \in X, \tau_{\varepsilon,W}(x) = \inf \set{t \ge \varepsilon}{g^t(x) \in W} \in \itv[cc]{\varepsilon}{+\infty} \eqpunct{.} \]
We can now derive a Kac's lemma for flows by applying Lemma \ref{lem:kaccompute} to the $\varepsilon$-time of the flow.
In particular, the $\varepsilon$ "margin of error" in the definition of the $\varepsilon$-hitting time allow us to ignore any possible pathological behaviour at the boundary of $W$.
Although this proposition is the simplest analogue of Kac's lemma for flow that one can devise, we could not find any explicit reference to this technique in the literature.

\begin{prop}[Kac's lemma for flows]
 \label{lem:kac}
 Let $X$ be a locally compact topological space, $\mathcal{B}$ its Borel $\sigma$-algebra, $(g^t)$ a continuous flow on $X$, and $\mu$ a Borel $(g^t)$-invariant Radon measure on $X$, which is ergodic and conservative.
 Let $W$ be an open relatively compact subset of $X$ with positive $\mu$-measure.
 Then for every $\varepsilon > 0$, the set
 \[ W_\varepsilon = \bigcup_{s \in \itv[co]{0}{\varepsilon}} g^{-s}(W) \]
 is open, relatively compact in $X$, the $\varepsilon$-hitting time map $\tau_{\varepsilon,W}$ is finite $\mu$-almost everywhere on $W_\varepsilon$ and
 \[ \sum_{k \ge 0} k \mu\( \left\{ \tau_{\varepsilon,W} \in \itv[co]{\varepsilon k}{\varepsilon(k+1)} \right\} \cap W_\varepsilon \) = \mu(X) \eqpunct{.} \]
\end{prop}

\begin{proof}
 Fix $\varepsilon > 0$, and let
 \[ X_{W,k} = \set{x \in X}{\tau_{\varepsilon,W}(x) \in \itv[co]{\varepsilon k}{\varepsilon(k+1)}} \]
 be the set of points of the whole space $X$ that $\varepsilon$-hit $W$ after a time approximately $\varepsilon k$.
 Denote by $f = g^\varepsilon$ the time-$\varepsilon$ map of the flow.
 We have
 \begin{align*}
  x \in X_{W,k} &\Leftrightarrow \exists t \in \itv[co]{\varepsilon k}{\varepsilon(k+1)}, g^t(x) \in W \eqand \forall s \in \itv[co]{\varepsilon}{\varepsilon k}, g^s(x) \not\in W \\
                &\Leftrightarrow \exists u \in \itv[co]{0}{\varepsilon}, f^k(x) \in g^{-u}(W) \\
                &\hspace{1.5em} \eqand \forall 1 \le i \le k-1, \forall s \in \itv[co]{0}{\varepsilon}, f^i(x) \not\in g^{-s}(W) \\
                &\Leftrightarrow f^k(x) \in W_\varepsilon \eqand \forall 0 \le i \le k-1, f^i(x) \not\in W_\varepsilon
 \end{align*}
 with $W_\varepsilon$ as given in the statement.
 This set $W_\varepsilon$ is open, hence measurable.
 Moreover, since $W$ is relatively compact and the flow is continuous, $W_\varepsilon$ is also relatively compact, and it contains $W$ so it has positive $\mu$-measure.
 Thus
 \[ X_{W,k} \cap W_\varepsilon = \set{x \in W_\varepsilon}{T_{W_\varepsilon}(x) = k} \]
 with $T_{W_\varepsilon}$ the $f$-return time map to $W_\varepsilon$ as defined before.
 Furthermore, if we define
 \[ X_{W,\infty} = \set{x \in X}{\tau_{\varepsilon,W}(x) = +\infty} \eqpunct{,} \]
 then the family of sets $(X_{W,k})_{k \ge 1}$ together with $X_{W,\infty}$ form a partition of $X$ and
 \[ X_{W,\infty} \cap W_\varepsilon = \set{x \in W_\varepsilon}{T_{W_\varepsilon}(x) = +\infty} \eqpunct{.} \]
 Observe that $\mu(X_{W,\infty} \cap W_\varepsilon) = 0$.
 Indeed, the conservativity of $(g^t)$ ensures that $\mu$-almost every point of $W$ will return to $W$.
 Therefore, by definition of $W_\varepsilon$, $\mu$-almost every point of $W_\varepsilon$ will return to $W_\varepsilon$ at some time multiple of $\varepsilon$.
 We can now apply Lemma \ref{lem:kaccompute} to $f$ and $W_\varepsilon$ to obtain that
 \[ \mu\( \bigcup_{k \in \Z} f^k(W_\varepsilon) \) = \sum_{k \ge 0} k \mu\( \left\{ T_{W_\varepsilon} = k \right\} \) = \sum_{k \ge 0} k \mu\( \left\{ \tau_{\varepsilon,W} \in \itv[co]{\varepsilon k}{\varepsilon(k+1)} \right\} \cap W_\varepsilon \) \eqpunct{,} \]
 where
 \[ \bigcup_{k \in \Z} f^k(W_\varepsilon) = \bigcup_{s \in \R} g^s(W) \]
 is a $(g^t)$-invariant measurable set which contains $W$ with $\mu(W) > 0$.
 Since $\mu$ is ergodic, this set has full measure.
\end{proof}


\subsection{Positive recurrence implies finiteness of the Gibbs measure}

The strategy of the proof of the theorem is very natural and simple.
We approximate the return time level sets $\{ \tau_{\varepsilon,W} \in \itv[co]{\varepsilon k}{\varepsilon(k+1)} \} \cap W_\varepsilon$ appearing in the above proposition by some products of shadows on the boundary, whose $m_F$ measure, thanks to Mohsen's Shadow Lemma, can be expressed in terms of exponentials of integrals of $F$ between points of $\Gamma x$.
However, as the length of the next subsections shows, the rigorous proof of this result is technically much more involved than the intuitive idea.
Recall that we only need to prove this result for $P(F) = 0$.

\begin{lemm}
 \label{lem:masstubemaj}
 Let $M$ be a negatively curved orbifold, and $\fn{F}{T^1 M}{\R}$ a H\"older continuous potential such that $P(F) = 0$.
 For all $x \in \wt{M}$ and $R \ge 0$, there exist $C \ge 0$ and a finite set $S \subset \Gamma$ such that for all $\gamma \in \Gamma \setminus S$, we have
 \[ \wt{m}_F\( \mathcal{U}_{x,R,\gamma} \) \le C e^{\int_x^{\gamma x} \wt{F}} \eqpunct{.} \]
\end{lemm}

\begin{proof}
 According to Lemma \ref{lem:shadowup}, there exists $l_0 \ge 0$ such that for all $\gamma \in \Gamma$ satisfying $d(x, \gamma x) \ge l_0$, we have in the Hopf coordinates
 \[ \mathcal{U}_{x,R,\gamma} \subset \O_{\gamma x} B(x, R+1) \times \O_x B(\gamma x, R+1) \times \set{\tau_x(v)}{v \in \mathcal{U}_{x,R,\gamma}} \eqpunct{.} \]
 Given any two vectors $v, w \in \mathcal{U}_{x,R,\gamma}$, observe that both $p_x(v)$ and $p_x(w)$ are in $B(x, R)$, so that $d(p_x(v), p_x(w)) \le 2 R$.
 This implies that $|\tau_x(v) - \tau_x(w)| \le 2R$ and
 \[ \mathcal{U}_{x,R,\gamma} \subset \O_{\gamma x} B(x, R+1) \times \O_x B(\gamma x, R+1) \times I_{2 R} \eqpunct{,} \]
 where $I_{2 R}$ is some interval of $\R$ of length $2 R$.

 We may assume that $l_0$ is large enough so that Lemma \ref{lem:shadowdisjoint} is satisfied for balls of radius $R+1$.
 Let $S = \set{\gamma \in \Gamma}{d(x, \gamma x) < l_0}$.
 Fix $\gamma \in \Gamma \setminus S$.
 The set $\O_{\gamma x} B(x, R+1) \times \O_x B(\gamma x, R+1)$ is relatively compact in $\partial_\infty \wt{M} \times \partial_\infty \wt{M} \setminus \Diag$, so that by continuity and positivity of the gap map, there exists a constant $C_1 \ge 0$ such that
 \[ \forall (\xi, \eta) \in \O_{\gamma x} B(x, R+1) \times \O_x B(\gamma x, R+1), \frac{1}{D_{F,x}(\xi, \eta)^2} \le C_1 \eqpunct{.} \]
 Equation \eqref{eqn:gibbsdef} defining the Gibbs measure $\wt{m}_F$ implies therefore
 \[ \wt{m}_F(\mathcal{U}_{x,R,\gamma}) \le C_1 \mu_x^{F \circ \iota}(\O_{\gamma x} B(x, R+1)) \mu_x^F(\O_x B(\gamma x, R+1)) 2 R \eqpunct{.} \]
 On the one hand $\mu_x^{F \circ \iota}$ is a probability measure on $\partial_\infty \wt{M}$, so $\mu_x^{F \circ \iota}(\O_{\gamma x} B(x, R+1)) \le 1$.
 On the other hand, Mohsen's Shadow Lemma \ref{lem:shadowmohsen} ensures the existence of a constant $C_2$ such that
 \[ \mu_x^F(\O_x B(\gamma x, R+1)) \le C_2 e^{\int_x^{\gamma x} \wt{F}} \eqpunct{.} \]
 The lemma is proved with $C = C_1 C_2 2 R$.
\end{proof}

\begin{lemm}
 \label{lem:recorbitlift}
 Let $W$ be an open relatively compact subset of $M$.
 Fix an open relatively compact lift $\wt{W}$ of $W$ to $\wt{M}$, and set $\mathcal{W} = T^1 W \subset T^1 M$ and $\wt{\mathcal{W}} = T^1 \wt{W} \subset T^1 \wt{M}$.
 For all $\varepsilon > 0$ and $x \in \wt{M}$, there exist $R \ge 0$ and $G \subset \Gamma$ finite such that the following holds.

 For every $v \in \mathcal{W}$ such that $\tau_{\varepsilon,\mathcal{W}}(v) \in \itv[co]{\varepsilon k}{\varepsilon(k+1)}$ with $k \ge 3$, and any lift $\wt{v}$ of $v$ to $\wt{\mathcal{W}}$, there exist $g, h \in G$ and $\gamma \in \Gamma_{\wt{W}}$ such that $\wt{v} \in \mathcal{U}_{\wt{\mathcal{W}}, h \gamma g}$ and $\varepsilon k - R \le d(x, \gamma x) \le \varepsilon (k + 1) + R$.
\end{lemm}

\begin{proof}
 Since $\wt{W}$ is relatively compact, there exists $R \ge 0$ such that $\wt{W} \subset B(x, R)$.
 The condition $\tau_{\varepsilon,\mathcal{W}}(v) \in \itv[co]{\varepsilon k}{\varepsilon(k+1)}$ with $k \ge 3$ means that there exists $T \in \itv[co]{\varepsilon k}{\varepsilon(k+1)}$ such that
 \[ v \in \mathcal{W} \eqand g^T(v) \in \mathcal{W} \eqand \forall s \in \itv[cc]{\varepsilon}{T-\varepsilon}, g^s(v) \not\in \mathcal{W} \eqpunct{,} \]
 where $0 < \varepsilon < T - \varepsilon < T$.
 Lift everything to $\wt{M}$.
 If $z_t = \pi(g^t(\wt{v}))$, observe that
 \[ z_0 \in \wt{W} \eqand z_T \in \gamma_0 \wt{W} \eqand \forall s \in \itv[cc]{\varepsilon}{T-\varepsilon}, z_s \not\in \Gamma \wt{W} \eqpunct{,} \]
 for some $\gamma_0 \in \Gamma$.
 In particular, $\wt{v} \in \mathcal{U}_{\wt{\mathcal{W}},\gamma_0}$.

 Let us show now that we can replace $\gamma_0$ by an element of $\Gamma_{\wt{W}}$.
 Define
 \[ G = \set{g \in \Gamma}{\exists y, y' \in \overline{\wt{W}}, d(y, g y') \le \varepsilon} \eqpunct{,} \]
 and note that
 \[ \gamma_0 G \gamma_0^{-1} = \set{g \in \Gamma}{\exists y, y' \in \overline{\gamma_0 \wt{W}}, d(y, g y') \le \varepsilon} \eqpunct{.} \]
 Remark also that
 \[ g \wt{W} \cap \itv[cc]{z_0}{z_\varepsilon} \neq \emptyset \Rightarrow g \in G \eqand g \wt{W} \cap \itv[cc]{z_{T-\varepsilon}}{z_T} \neq \emptyset \Rightarrow g \in \gamma_0 G \gamma_0^{-1} \eqpunct{.} \]
 Set $I = \set{s \in \itv[cc]{0}{\varepsilon}}{z_s \in G \wt{W}}$.
 There exists $u \in I$ such that for every $g \in G$ with $z_u \in g \wt{W}$, we have
 \[ \forall h \in \Gamma, h \wt{W} \cap \itv[cc]{z_u}{z_\varepsilon} \neq \emptyset \Rightarrow h \overline{\wt{W}} \cap g \overline{\wt{W}} \neq \emptyset \eqpunct{.} \]
 Indeed, otherwise we could find for every $u \in I$ elements $g_u \in G$ and $h_u \in \Gamma$ such that $z_u \in g_u \wt{W}$ and
 \[ h_u \wt{W} \cap \itv[cc]{z_u}{z_\varepsilon} \neq \emptyset \eqand h_u \overline{\wt{W}} \cap g_u \overline{\wt{W}} = \emptyset \eqpunct{.} \]
 In particular, note that $h_u \in G$ since $h_u \wt{W}$ meets $\itv[cc]{z_0}{z_\varepsilon}$.
 Denote by $u_\infty = \sup{I} \in \overline{I}$.
 Since $G$ is finite, we can take an increasing sequence of $(u_n)_n$ converging to $u_\infty$ such that $g_{u_n} = g$ and $h_{u_n} = h$ for every $n$.
 For every $n$, there also exists $v_n \in \itv[cc]{u_n}{\varepsilon}$ such that $z_{v_n} \in h \wt{W}$.
 As $h \in G$, $v_n \in I$ hence $u_n \le v_n \le u_\infty$.
 Therefore both $(u_n)_n$ and $(v_n)_n$ converge to $u_\infty$, and taking the limit as $n$ goes to infinity yields
 \[ z_{u_\infty} \in h \overline{\wt{W}} \cap g \overline{\wt{W}} \eqpunct{.} \]
 This is a contradiction.
 Likewise, there exists $v \in \set{s \in \itv[cc]{T-\varepsilon}{T}}{z_s \in \gamma_0 G \gamma_0^{-1} \wt{W}}$ such that for every $g \in \gamma_0 G \gamma_0^{-1}$ for which $z_v \in g \gamma_0 \wt{W}$ we have
 \[ \forall h \in \Gamma, h \wt{W} \cap \itv[cc]{z_{T-\varepsilon}}{z_v} \neq \emptyset \Rightarrow h \overline{\wt{W}} \cap g \overline{\gamma_0 \wt{W}} \neq \emptyset \eqpunct{.} \]

 By definition of these $u$ and $v$, one can find $g_u, g_v \in G$ such that
 \[ z_u \in g_u \wt{W} \eqand z_v \in (\gamma_0 g_v \gamma_0^{-1}) \gamma_0 \wt{W} \eqpunct{.} \]
 Let $\gamma_1 = (\gamma_0 g_v \gamma_0^{-1}) \gamma_0 g_u^{-1} = \gamma_0 g_v g_u^{-1}$.
 Note how $z_v \in \gamma_1 g_u \wt{W}$.
 The previous discussion ensures that $\gamma_1 \in \Gamma_{g_u \wt{W}} = g_u \Gamma_{\wt{W}} g_u^{-1}$, therefore
 \[ \gamma = g_u^{-1} \gamma_1 g_u = g_u^{-1} \gamma_0 g_v \in \Gamma_{\wt{W}} \eqpunct{,} \]
 and we still have $\wt{v} \in \mathcal{U}_{\wt{\mathcal{W}}, g_u \gamma g_v^{-1}}$.

 Finally, the triangle inequality gives that
 \begin{align*}
  \left| d(x, \gamma x) - T \right| &= \left| d(x, g_u^{-1} \gamma_0 g_v x) - d(z_0, z_T) \right| \\
     &\le \left| d(g_u x, \gamma_0 g_v x) - d(z_u, z_v) \right| + 2 \varepsilon \\
     &\le d(g_u x, z_u) + d(\gamma_0 g_v x, z_v) + 2 \varepsilon < 2 R + 2 \varepsilon
 \end{align*}
 since $z_u \in g_u \wt{W}$ and $z_v \in \gamma_0 g_v \wt{W}$.
 Therefore $\varepsilon k - R' \le d(x, \gamma x) \le \varepsilon (k + 1) + R'$ with $R' = 2 R + 2 \varepsilon$.
\end{proof}

\begin{prop}
 \label{prop:majtheocrit1}
 Let $\wt{W}$ be an open relatively compact subset of $\wt{M}$ such that $T^1 \wt{W}$ meets the nonwandering set $\wt{\Omega}$ of $\tilde{g}^t$.
 If $(\Gamma, \wt{F})$ is positive recurrent relatively to $\wt{W}$, then $m_F$ is finite.
\end{prop}

\begin{proof}
 Let $W = \Pr(\wt{W})$, $\mathcal{W} = T^1 W$ and $z \in \wt{M}$ from the positive recurrence property.
 Choose $R$ such that $\wt{W} \subset B(z, R)$.
 Fix $\varepsilon > 0$.
 If $v \in \mathcal{W}_\varepsilon$, there exists $s \in \itv[co]{0}{\varepsilon}$ such that $v' = g^s(v) \in \mathcal{W}$.
 Now if $\tau_{\varepsilon,\mathcal{W}}(v) \in \itv[co]{\varepsilon k}{\varepsilon(k+1)}$, $k \ge 4$, then
 \[ \tau_{\varepsilon,\mathcal{W}}(v') \in \itv[co]{\varepsilon k - s}{\varepsilon(k+1)-s} \subset \itv[co]{\varepsilon(k-1)}{\varepsilon k} \sqcup \itv[co]{\varepsilon k}{\varepsilon(k+1)} \eqpunct{.} \]
 Since we assumed that $k-1 \ge 3$, Lemma \ref{lem:recorbitlift} gives the existence of $R \ge 0$ and $G \subset \Gamma$ finite, both independent of $v$, such that for every lift $\wt{v}'$ of $v'$ to $\wt{\mathcal{W}} = T^1 \wt{W}$ there are $\gamma \in \Gamma_{\wt{W}}$ and $g, h \in G$ such that
 \[ \varepsilon(k-1) - R \le d(z, \gamma z) \le \varepsilon(k+1) + R \eqand \wt{v}' \in \mathcal{U}_{\wt{\mathcal{W}},h \gamma g} \eqpunct{.} \]
 Therefore, any lift $\wt{v}$ of $v$ to $\wt{\mathcal{W}}$ will satisfy
 \[ \wt{v} \in \bigcup_{\gamma \in \Gamma_{\wt{W},k}} \bigcup_{g, h \in G} \bigcup_{s \in \itv[co]{0}{\varepsilon}} \wt{g}^{-s}\( \mathcal{U}_{\wt{\mathcal{W}},h \gamma g} \) \subset \bigcup_{\gamma \in \Gamma_{\wt{W},k}} \bigcup_{g, h \in G} \mathcal{U}_{z,R+\varepsilon,h \gamma g} \eqpunct{,} \]
 where $\Gamma_{\wt{W},k} = \set{\gamma \in \Gamma_{\wt{W}}}{\varepsilon(k-1) - R \le d(z, \gamma z) \le \varepsilon(k+1) + R}$.
 This ensures that
 \[ \forall k \ge 4, m_F\(\left\{ \tau_{\varepsilon,\mathcal{W}} \in \itv[co]{\varepsilon k}{\varepsilon(k+1)} \right\} \cap \mathcal{W}_\varepsilon \) \le \sum_{\gamma \in \Gamma_{\wt{W},k}} \sum_{g, h \in G} \wt{m}_F\( \mathcal{U}_{z,R+\varepsilon,h \gamma g} \) \eqpunct{.} \]

 According to Lemma \ref{lem:masstubemaj}, there exists $S$ finite and $C \ge 0$ (both depending on $z$, $R$ and $\varepsilon$) such that
 \[ \forall k \ge 4, m_F\(\left\{ \tau_{\varepsilon,\mathcal{W}} \in \itv[co]{\varepsilon k}{\varepsilon(k+1)} \right\} \cap \mathcal{W}_\varepsilon \) \le C \sum_{\gamma \in \Gamma_{\wt{W},k}} \sum_{\substack{g, h \in G \\ h \gamma g \not\in S}} e^{\int_z^{h \gamma g z} \wt{F}} \eqpunct{.} \]
 By Lemma \ref{lem:potcontrol}, there is a constant $C' \ge 0$ which only depends on $\wt{F}$, the geometry of $\wt{M}$ and on $\sup \set{d(z, g z)}{g \in G} < +\infty$ such that
 \[ \forall \gamma \in \Gamma, \int_z^{h \gamma g z} \wt{F} = \int_{h^{-1} z}^{\gamma g z} \wt{F} \le C' + \int_z^{\gamma z} \wt{F} \eqpunct{.} \]
 Hence if $K = \card G$ we get
 \[ \forall k \ge 4, m_F\(\left\{ \tau_{\varepsilon,\mathcal{W}} \in \itv[co]{\varepsilon k}{\varepsilon(k+1)} \right\} \cap \mathcal{W}_\varepsilon \) \le C K^2 e^{C'} \sum_{\gamma \in \Gamma_{\wt{W},k}} e^{\int_z^{\gamma z} \wt{F}} \eqpunct{.} \]
 For all $k \ge k_0 = \max(6, \frac{2 R}{\varepsilon})$ and all $\gamma \in \Gamma_{\wt{W},k}$, we have
 \[ d(z, \gamma z) \ge \varepsilon(k-1) - R \ge \varepsilon\( k - 1 - \frac{k}{2} \) \ge \frac{\varepsilon}{3} k \eqpunct{.} \]
 Since $(\Gamma, F)$ is divergent, $m_F$ is ergodic and conservative (see Hopf-Tsuji-Sullivan \cite[Theorem 5.4]{PPS}).
 Moreover, $\mathcal{W}$ meets $\Omega$ hence has positive $m_F$-measure.
 We can apply Lemma \ref{lem:kac} to obtain
 \[ m_F(T^1 M) \le A + \frac{3 C K^2 e^{C'}}{\varepsilon} \sum_{k \ge k_0} \sum_{\gamma \in \Gamma_{\wt{W},k}} \sum_{g, h \in G} d(z, \gamma z) e^{\int_z^{\gamma z} \wt{F}} \eqpunct{,} \]
 where $A$ is the finite sum of the first $k_0 - 1$ terms in Kac's lemma.

 Finally, note that for all $\gamma \in \Gamma_{\wt{W}}$, $\gamma \in \Gamma_{\wt{W},k}$ if and only if
 \[ \frac{d(z, \gamma z) - R}{\varepsilon} - 1 \le k \le \frac{d(z, \gamma z) + R}{\varepsilon} + 1 \eqpunct{,} \]
 which allows at most $2 \frac{R}{\varepsilon} + 3$ possibilities.
 Therefore
 \[ m_F(T^1 M) \le A + \frac{3 C K^2 e^{C'}}{\varepsilon} \( 2 \frac{R}{\varepsilon} + 3 \) \sum_{\gamma \in \Gamma_{\wt{W}}} d(z, \gamma z) e^{\int_z^{\gamma z} \wt{F}} \eqpunct{,} \]
 and the positive recurrence of $(\Gamma, \wt{F})$ with respect to $\wt{W}$ implies that $m_F$ is finite.
\end{proof}


\subsection{Finiteness of the Gibbs measure implies positive recurrence}

The aim of this section is to prove the following proposition which, combined together with Proposition \ref{prop:majtheocrit1}, will prove Theorem \ref{thm:crit1}.
The idea of the proof, as said earlier, is very natural, even if the rigorous details take a long time to be written.

\begin{prop}
 \label{prop:mintheocrit1}
 Let $M$ be a negatively curved orbifold and $\fn{F}{T^1 M}{\R}$ a H\"older continuous potential such that $P(F) = 0$.
 If $m_F$ is finite, then $(\Gamma, \wt{F})$ is positive recurrent relatively to any open relatively compact subset $\wt{W} \subset \wt{M}$ that intersects $\pi(\wt{\Omega})$.
\end{prop}

Recall that when $m_F$ is finite, it is ergodic, so that $(\Gamma, F)$ is divergent.

\begin{lemm}
 \label{lem:masstubemin}
 For every $x \in \wt{M}$ such that $B(x, R)$ intersects $\pi(\wt{\Omega})$, there exist $C > 0$ and $S, G \subset \Gamma$ finite such that for all $\gamma \in \Gamma \setminus S$, there exist $g, h \in G$ such that
 \[ \wt{m}_F\( \mathcal{U}_{g x,R,\gamma h g^{-1}} \) \ge C e^{\int_x^{\gamma x} \wt{F}} \eqpunct{.} \]
\end{lemm}

\begin{proof}
 The proof of this lemma is similar to the one of Lemma \ref{lem:masstubemaj}, but we will use Lemma \ref{lem:shadowprod} instead of Lemma \ref{lem:shadowmohsen}.

 Take $\varepsilon > 0$ such that we still have $B(x, R-2\varepsilon) \cap \pi(\wt{\Omega}) \neq \emptyset$.
 Lemma \ref{lem:shadowprod} applied to $B(x, R-2\varepsilon)$ gives us $C_1 > 0$ and $S, G \subset \Gamma$ finite such that for all $\gamma \in \Gamma$ there exist $g, h \in G$ such that
 \[ \frac{1}{C_1} e^{\int_x^{\gamma x} \wt{F}} \le \mu_x^{F \circ \iota}(\O_{\gamma x} B(g x, R-2\varepsilon)) \mu_x^F(\O_x B(\gamma h x, R-2\varepsilon) \le C_1 e^{\int_x^{\gamma x} \wt{F}} \eqpunct{.} \]
 Thanks to Lemma \ref{lem:shadowcontrol} applied for $\delta = \sup \set{d(x, g x)}{g \in G}$, there exists $l_0 \ge 0$ such that if $d(x, \gamma x) \ge l_0$ then
 \[ \O_{\gamma x} B(g x, R-2\varepsilon) \times \O_x B(\gamma h x, R-2\varepsilon) \subset \O_{\gamma h x} B(g x, R-\varepsilon) \times \O_{g x} B(\gamma h x, R-\varepsilon) \eqpunct{.} \]
 But according to Lemma \ref{lem:shadowdown} (applied with $\frac{\varepsilon}{2}$ instead of $\varepsilon$), we can assume that $l_0$ is large enough that for all $\gamma \in \Gamma$ satisfying $d(x, \gamma x) \ge l_0$, one has
 \[ \O_{\gamma h x} B(g x, R-\varepsilon) \times \O_{g x} B(\gamma h x, R-\varepsilon) \times \sing{0} \subset \mathcal{U}_{g x,R-\frac{\varepsilon}{2},\gamma h g^{-1}} \eqpunct{.} \]
 Indeed, for every $v_- \in \O_{\gamma h x} B(g x, R-\varepsilon)$ and $v_+ \in \O_{g x} B(\gamma h x, R-\varepsilon)$, the geodesic $(v_- v_+)$ intersects $B(g x,R-\frac{\varepsilon}{2})$, so that $v = (v_-, v_+, 0)$ always belongs to $\mathcal{U}_{g x,R-\frac{\varepsilon}{2},\gamma h g^{-1}}$.
 This implies easily that
 \[ \O_{\gamma h x} B(g x, R-\varepsilon) \times \O_{g x} B(\gamma h x, R-\varepsilon) \times \itv[oo]{-\frac{\varepsilon}{2}}{\frac{\varepsilon}{2}} \subset \mathcal{U}_{g x,R,\gamma h g^{-1}} \eqpunct{.} \]

 We may assume that $l_0$ is large enough so that Lemma \ref{lem:shadowdisjoint} is satisfied for balls of radius $R-\varepsilon$.
 By possibly adding finitely many elements to $S$, we may also assume that $d(x, \gamma x) \ge l_0$.
 The set $\O_{\gamma h x} B(g x, R-\varepsilon) \times \O_{g x} B(\gamma h x, R-\varepsilon)$ is relatively compact in $\partial_\infty \wt{M} \times \partial_\infty \wt{M} \setminus \Diag$, thus continuity and positivity of the gap map ensure the existence of $C_2 > 0$ such that
 \[ \forall (\xi, \eta) \in \O_{\gamma x} B(g x, R-\varepsilon) \times \O_{g x} B(\gamma h x, R-\varepsilon), \frac{1}{D_{F,x}(\xi, \eta)^2} \ge C_2 \eqpunct{.} \]
 Therefore, by definition \eqref{eqn:gibbsdef} of the Gibbs measure $\wt{m}_F$, we have
 \begin{equation*}
  \wt{m}_F(\mathcal{U}_{g x,R,\gamma h g^{-1}}) \ge C_2 \mu_x^{F \circ \iota}(\O_{\gamma h x} B(g x, R-\varepsilon)) \mu_x^F(\O_{g x} B(\gamma h x, R-\varepsilon)) \ge \frac{C_2}{C_1} e^{\int_x^{\gamma x} \wt{F}} \eqpunct{.} \qedhere
 \end{equation*}
\end{proof}

For $\gamma \in \Gamma$, define
\[ \mathcal{E}_{x,R,\gamma,\varepsilon} = \set{v \in T^1 B(x, R)}{\begin{cases} \exists t > 2 \varepsilon, &g^t(v) \in T^1 B(\gamma x, R) \\ \forall s \in \itv[cc]{\varepsilon}{t-\varepsilon}, &g^s(v) \not\in \Gamma T^1 B(x, R) \end{cases}} \eqpunct{.} \]
The following lemma ensures that a geodesic ball $\mathcal{U}_{x,R,\gamma}$ with $\gamma$ "close" from $\Gamma_{B(x, R)}$ is contained in a compact union of sets of the type $\mathcal{E}_{x,R-\varepsilon,\gamma',\varepsilon}$ which, as we will see later, project to vectors with $\epsilon$-return time into $T^1 W$ comparable with $d(x, \gamma x)$.

\begin{lemm}
 \label{lem:tubesplit}
 Assume that $B(x, R) \subset \wt{W}$ with $x \in \wt{M} \setminus \wt{\Sigma}$.
 For every $\varepsilon \in \itv[oo]{0}{R}$ and every $D \ge 0$, there exist finite subsets $S, G \subset \Gamma$ and $\theta > 0$ such that for every $\gamma_0 \in \Gamma_{\wt{W}} \setminus S$, every $h \in \Gamma$ such that $d(x, h x) \le D$, and every $\gamma \in \Gamma$ satisfying $d(\gamma_0 x, \gamma h x) \le D$ we have
 \[ \mathcal{U}_{h x,R-\varepsilon,\gamma} \subset \bigcup_{g \in G} \bigcup_{\gamma' \in \gamma_0 G g^{-1}} \bigcup_{s \in \itv[cc]{0}{\theta}} g^{-s}\( \mathcal{E}_{g x, R-\varepsilon, \gamma', \varepsilon} \) \eqpunct{.} \]
\end{lemm}

\begin{proof}
 Since $\wt{W}$ is relatively compact and $H$ is finite, there is $R_1 \ge \max(R, D)$ such that $B(x, R) \subset \wt{W} \subset B(x, R_1)$.

 By definition of $\gamma_0 \in \Gamma_{\wt{W}}$, there exist $y, y' \in \wt{W} \subset B(x, R_1)$ such that, if $\itv[cc]{y}{\gamma_0 y'}$ intersects some $g \wt{W}$, then $g \overline{\wt{W}}$ intersects either $\overline{\wt{W}}$ or $\overline{\gamma_0 \wt{W}}$.
 In particular, since $B(g x, R) \subset g \wt{W} \subset B(g x, R_1)$, then either $d(y, g x) \le 2 R_1$ or $d(\gamma_0 y', g x) \le 2 R_1$.

 Let $\varepsilon \in \itv[oo]{0}{R}$, and $\rho_\varepsilon = \rho_1(2 R, 3 R_1) - \log \varepsilon$ given by Lemma \ref{lem:lrsa} for the open relatively compact subsets $B(x, R - \varepsilon) \subset B(x, R)$.
 By possibly increasing it, we may assume that $\rho_\varepsilon \ge 3 R_1$.
 Therefore, a ball $B(g x, R)$ at distance greater than $\rho_\varepsilon$ both from $y$ and $y'$ cannot intersect $\itv[cc]{y}{\gamma_0 y'}$.
 Lemma \ref{lem:lrsa} ensures that for such $g \in \Gamma$ the ball $B(g x, R - \varepsilon)$ does not meet any segment $\itv[cc]{z}{z'}$ with $d(y, z), d(\gamma_0 y', z') \le 3 R_1$.

 Now take $h$ and $\gamma$ such that $d(x, h x) \le D$ and $d(\gamma_0 x, \gamma h x) \le D$.
 If $v \in \mathcal{U}_{h x,R,\gamma}$, we have for some $T \ge 0$
 \[ d(\pi(v), y) \le R_1 + D + R_1 \le 3 R_1 \eqand d(\pi(g^T(v)), \gamma_0 y') \le 3 R_1 \eqpunct{.} \]
 Thus, for $0 \le t \le T$, $g^t(v) \in T^1 B(g x, R-\varepsilon)$ only if $d(y, B(g x, R)) \le \rho_\varepsilon$ or $d(\gamma_0 y', B(g x, R)) \le \rho_\varepsilon$, which implies respectively
 \[ d(x, g x) \le R_1 + \rho_\varepsilon + R \eqor d(\gamma_0 x, g x) \le R_1 + \rho_\varepsilon + R \eqpunct{.} \]
 We denote by $G = \set{g \in \Gamma}{d(x, g x) \le R_1 + R + \rho_\varepsilon}$.
 Observe that
 \[ \set{g \in \Gamma}{d(\gamma_0 x, g x) \le R_1 + R + \rho_\varepsilon} = \gamma_0 G \eqpunct{,} \]
 hence both sets are finite and have same cardinal.

 Let $v \in \mathcal{U}_{h x,R-\varepsilon,\gamma}$, i.e. $v \in T^1 B(h x, R-\varepsilon)$ and $g^{t_v}(v) \in T^1 B(\gamma h x, R-\varepsilon)$ for some $t_v \ge 0$ satisfying
 \[ t_v = d(\pi(v), \pi(g^{t_v}(v))) \ge d(h x, \gamma h x) - 2 (R - \varepsilon) \ge d(x, \gamma_0 x) - 2 (D + R - \varepsilon) \eqpunct{.} \]
 We recall that $g^s(v) \in T^1 B(g x, R - \varepsilon)$ may only happen for $g \in G \cup \gamma_0 G$.
 If for example $g \in G$, note that
 \[ s = d(\pi(v), \pi(g^s(v))) \le d(h x, g x) + 2 (R - \varepsilon) \le D + R_1 + R + \rho_\varepsilon + 2 (R - \varepsilon) = \theta \eqpunct{.} \]
 Since every ball of $\wt{M}$ is convex, the set
 \[ I_v = \set{t \in \itv[oo]{0}{t_v}}{g^t(v) \in \Gamma T^1 B(x, R - \varepsilon)} \]
 is open and made of finitely many connected components included in $\itv[oo]{0}{\theta} \cup \itv[oo]{t_v - \theta}{t_v}$.
 Let
 \[ S = \set{\gamma_0 \in \Gamma}{d(x, \gamma_0 x) \le 2 \theta + 2 (D + R - \varepsilon)} \eqpunct{,} \]
 so that $\gamma_0 \not\in S$ implies $t_v > 2 \theta$ and these two intervals are disjoint.
 Furthermore, we have that
 \begin{enumerate}
  \item[$(i)$\:]   there exists $s_v \in \itv[oo]{0}{\theta}$ such that $g^{s_v}(v) \in T^1 B(g_v x, R - \varepsilon)$ for some $g_v \in G$ ;
  \item[$(ii)$\:]  there exists $u_v \in \itv[oo]{t_v - \theta}{t_v}$ such that $g^{u_v}(v) \in T^1 B(h_v x, R - \varepsilon)$ for some $h_v \in \gamma_0 G$ ;
  \item[$(iii)$\:] $g^t(v) \not\in \Gamma T^1 B(x, R - \varepsilon)$ for all $t \in \itv[cc]{s_v + \varepsilon}{u_v - \varepsilon}$.
 \end{enumerate}
 This means exactly that $g^{s_v}(v) \in \mathcal{E}_{g_v x,R-\varepsilon,\gamma',\varepsilon}$ where $\gamma'$ is an isometry mapping $g_v x$ to $h_v x$.
 Since $x \not\in \wt{\Sigma}$, $\gamma'$ satisfies $\gamma' g_v = h_v$, i.e. $\gamma' \in \gamma_0 G g_v^{-1}$.
 This concludes the proof.
\end{proof}

For all $k \ge 1$, we define
\[ \Gamma_{\wt{W},k} = \set{\gamma \in \Gamma_{\wt{W}}}{d(x, \gamma x) \in \itv[co]{k-1}{k}} \eqpunct{.} \]

\begin{lemm}
 \label{lem:bigeprojmeas}
 Let $R, \varepsilon > 0$ and $G \subset \Gamma$ finite.
 There are two constants $C, L > 0$ such that the following holds : for every $g \in G$ and every $k \ge 1$
 \begin{align*}
  \wt{m}_F&\( \bigcup_{\gamma \in \Gamma_{\wt{W},k}} \bigcup_{\gamma' \in \gamma G g^{-1}} \bigcup_{s \in \itv[co]{0}{\varepsilon}} g^{-s}\( \mathcal{E}_{g x,R,\gamma',\varepsilon} \) \) \\
     &\hspace{3em}\le C m_F\( \bigcup_{s \in \itv[co]{0}{\varepsilon}} g^{-s}\( \left\{ \tau_{\varepsilon,T^1 B(\wt{\pi}(x),R)} \in \itv[co]{k - L}{k + 1 + L} \right\} \) \) \eqpunct{.}
 \end{align*}
\end{lemm}

\begin{proof}
 Take $v \in \mathcal{E}_{g x,R,\gamma',\varepsilon}$ with $\gamma \in \Gamma_{\wt{W},k}$, $g \in G$ and $\gamma' = \gamma h g^{-1}$ for some $h \in G$.
 This means that $v \in T^1 B(g x, R)$, $g^t(v) \in T^1 B(\gamma' g x, R)$ for some $t \ge 2 \varepsilon$, and
 \[ \forall s \in \itv[cc]{\varepsilon}{t-\varepsilon}, g^s(v) \not\in \Gamma T^1 B(x, R) \eqpunct{.} \]
 In particular, $t$ must satisfy $\left| t - d(g x, \gamma' g x) \right| \le 2 R $, which implies, by triangular inequality,
 \[ \left| t - d(x, \gamma x) \right| \le 2 (R + L_0) \eqpunct{,} \]
 where $L_0 = \sup\set{d(x, g x)}{g \in G}$.
 Denote by $\mathcal{W} = T^1 B(\Pr(x), R)$ the projection of $T^1 B(x, R)$ in $T^1 M$.
 Then $\Pr(v)$ satisfies
 \begin{enumerate}
  \item[$(i)$\:] $\Pr(v) \in \mathcal{W}$ ;
  \item[$(ii)$\:] $\Pr(g^t(v)) \in \mathcal{W}$ ;
  \item[$(iii)$\:] $\forall s \in \itv[cc]{\varepsilon}{t-\varepsilon}, \Pr(g^s(v)) \not\in \mathcal{W}$.
 \end{enumerate}
 Therefore
 \[ \tau_{\varepsilon, \mathcal{W}}(\Pr(v)) \in \itv[cc]{t-\varepsilon}{t} \subset \itv[co]{d(x, \gamma x) - L}{d(x, \gamma x) + L} \subset \itv[co]{k - L}{k + 1 + L} \]
 with $L = 2 (R + L_0) + \varepsilon - 1$.
 In other words,
 \[ \Pr\( \mathcal{E}_{g x,R,\gamma',\varepsilon} \) \subset \left\{ \tau_{\varepsilon,\mathcal{W}} \in \itv[co]{k - L}{k + 1 + L} \right\} \eqpunct{,} \]
 where the set on the right hand side depends neither on $\gamma$ nor on $\gamma'$.
 Hence
 \begin{align*}
  \bigcup_{\gamma \in \Gamma_{\wt{W},k}} \bigcup_{\gamma' \in \gamma G g^{-1}} &\bigcup_{s \in \itv[co]{0}{\varepsilon}} g^{-s}\( \mathcal{E}_{g x,R,\gamma',\varepsilon} \) \\
   &\subset \Pr^{-1} \( \bigcup_{s \in \itv[co]{0}{\varepsilon}} g^{-s}\( \left\{ \tau_{\varepsilon,\mathcal{W}} \in \itv[co]{k - L}{k + 1 + L} \right\} \) \) \eqpunct{.}
 \end{align*}
 The union on the left is relatively compact in $T^1 \wt{M}$, because included in $T^1 B(g x, R + \varepsilon)$.
 Therefore, any $w \in \bigcup_{s \in \itv[co]{0}{\varepsilon}} g^{-s}\( \left\{ \tau_{\varepsilon,\mathcal{W}} \in \itv[co]{k - L}{k + 1 + L} \right\}\)$ has finitely many preimages that lie in the union of the left hand side, their number being bounded from above by
 \[ C = \card \set{h \in \Gamma}{B(g x, R + \varepsilon) \cap B(h g x, R + \varepsilon) \neq \emptyset} \eqpunct{,} \]
 which is actually independent of $g$.
 As $m_F = {\Pr}_*(\tilde{m}_F)$, this concludes the proof of the lemma.
\end{proof}

\begin{lemm}
 \label{lem:bigufiniteint}
 For all $x \in \wt{M}$ and $R > 0$, there exists a constant $C \ge 0$ such that for every $k \ge 1$
 \[ \sum_{\gamma \in \Gamma_{\wt{W},k}} \wt{m}_F\( \mathcal{U}_{x,R,\gamma} \) \le C \wt{m}_F\( \bigcup_{\gamma \in \Gamma_{\wt{W},k}} \mathcal{U}_{x,R,\gamma} \) \eqpunct{.} \]
\end{lemm}

\begin{proof}
 Assume that $\gamma, \gamma' \in \Gamma_{\wt{W},k}$ are such that $\mathcal{U}_{x,R,\gamma} \cap \mathcal{U}_{x,R,\gamma'} \neq \emptyset$.
 There exist $v \in T^1 B(x, R)$ and $t, s \ge 0$ such that $g^t(v) \in T^1 B(\gamma x, R)$ and $g^s(v) \in T^1 B(\gamma' x, R)$.
 Without loss of generality, assume that $t \le s$.
 We have
 \[ d(\gamma x, \gamma' x) \le d(\gamma x, \pi(g^t(v))) + s - t + d(\pi(g^s(v)), \gamma' x) \le s - t + 2 R \eqpunct{.} \]
 But on the one hand,
 \[ t = d(\wt{\pi}(v), \wt{\pi}(g^t(v))) \ge d(x, \gamma x) - 2 R \ge k - 1 - 2 R \eqpunct{,} \]
 and on the other hand
 \[ s = d(\wt{\pi}(v), \wt{\pi}(g^t(v))) \le d(x, \gamma x) + 2 R \le k + 2 R \eqpunct{.} \]
 Therefore $d(\gamma x, \gamma' x) \le 1 + 6 R$ i.e. $\gamma' \in \gamma G$ where $G = \set{g \in \Gamma}{d(x, g x) \le 1 + 6 R}$ is finite.
 This ensures that
 \[ \sum_{\gamma \in \Gamma_{\wt{W},k}} \wt{m}_F\( \mathcal{U}_{x,R,\gamma} \) \le C \wt{m}_F\( \bigcup_{\gamma \in \Gamma_{\wt{W},k}} \mathcal{U}_{x,R,\gamma} \) \eqpunct{,} \]
 where $C = \card G$.
\end{proof}

We can now prove Proposition \ref{prop:mintheocrit1}.

\begin{proof}
 Recall that we assumed that $P(F) = 0$.
 Since $\wt{\Sigma}$ is locally finite, there exists $x \in \wt{W}$ and $0 < R_1 < R_2$ such that
 \[ T^1 B(x, R_1) \cap \wt{\Omega} \neq \emptyset \eqand \overline{B(x, R_1)} \subset B(x, R_2) \subset \wt{W} \eqpunct{.} \]
 Lemma \ref{lem:masstubemin} ensures that there exist finite sets $S_1, G_1 \subset \Gamma$ and a constant $C_1 > 0$ such that for all $\gamma \in \Gamma \setminus S_1$, there exist $g, h \in G_1$ such that
 \[ e^{\int_x^{\gamma x} \wt{F}} \le \frac{1}{C_1} \wt{m}_F\( \mathcal{U}_{g x,R_1,\gamma h g^{-1}} \) \eqpunct{.} \]
 Fix $k \ge 1$.
 We have
 \[ \sum_{\gamma \in \Gamma_{\wt{W},k} \setminus S_1} e^{\int_x^{\gamma x} \wt{F}} \le \frac{1}{C_1} \sum_{g, h \in G_1} \sum_{\gamma \in \Gamma_{\wt{W},k}} \wt{m}_F\( \mathcal{U}_{g x,R_1,\gamma h g^{-1}} \) \eqpunct{.} \]
 According to Lemma \ref{lem:bigufiniteint}, there is a constant $C_2 \ge 0$ which does not depend on $k$ such that
 \[ \sum_{\gamma \in \Gamma_{\wt{W},k} \setminus S_1} e^{\int_x^{\gamma x} \wt{F}} \le \frac{C_2}{C_1} \sum_{g, h \in G_1} \wt{m}_F\( \bigcup_{\gamma \in \Gamma_{\wt{W},k}} \mathcal{U}_{g x,R_1,\gamma h g^{-1}} \) \eqpunct{.} \]
 Observe that if $g, h \in G_1$ then
 \[ d(\gamma h g^{-1} (g x), \gamma x) = d(h x, x) \le D = \sup \set{d(x, g x)}{g \in G_1} \eqpunct{.} \]
 Therefore, Lemma \ref{lem:tubesplit} applied with $R = R_2$, $\varepsilon < R_2 - R_1$ and this $D$ gives the existence of finite sets $S_2, G_2 \subset \Gamma$ and of $\theta < N \varepsilon$ such that for every $g, h \in G_1$ one has
 \[ \forall \gamma \in \Gamma_{\wt{W}} \setminus S_2, \mathcal{U}_{g x,R_1,\gamma h g^{-1}} \subset \mathcal{U}_{g x,R_\varepsilon,\gamma h g^{-1}} \subset \bigcup_{g' \in G_2} \bigcup_{\gamma' \in \gamma G_2 {g'}^{-1}} \bigcup_{s \in \itv[cc]{0}{\theta}} g^{-s}\( \mathcal{E}_{g' x,R_\varepsilon,\gamma',\varepsilon} \) \eqpunct{,} \]
 with $R_\varepsilon = R_2 - \varepsilon > R_1$.
 Assume that $S_1 \subset S_2$.
 The $(g^t)$-invariance of $\wt{m}_F$ gives
 \[ \sum_{\gamma \in \Gamma_{\wt{W},k} \setminus S_2} e^{\int_x^{\gamma x} \wt{F}} \le M_k + C_3 \sum_{g \in G_2} \wt{m}_F\( \bigcup_{\gamma \in \Gamma_{\wt{W},k}} \bigcup_{\gamma' \in \gamma G_2 g^{-1}} \bigcup_{s \in \itv[co]{0}{\varepsilon}} g^{-s}\( \mathcal{E}_{g x,R_\varepsilon,\gamma',\varepsilon} \) \) \eqpunct{,} \]
 where
 \[ C_3 = \frac{C_2 N}{C_1} (\card G_1)^2 \eqand M_k = \frac{C_2}{C_1} \sum_{g, h \in G_1} \wt{m}_F\( \bigcup_{\gamma \in \Gamma_{\wt{W},k} \cap S_2} \mathcal{U}_{g x,R_\varepsilon,\gamma h g^{-1}} \) \eqpunct{.} \]
 Note that $M_k$ is finite, and even $M_k = 0$ for $k \ge k_0$ since $S_2$ is finite.

 Apply now Lemma \ref{lem:bigeprojmeas} with $R = R_\varepsilon$ to obtain two constants $C_4, L \ge 0$ such that
 \[ \sum_{\gamma \in \Gamma_{\wt{W},k} \setminus S_2} e^{\int_x^{\gamma x} \wt{F}} \le M_k + C_3 C_4 \card G_2 m_F\( \bigcup_{s \in \itv[co]{0}{\varepsilon}} g^{-s}\( \left\{ \tau_{\varepsilon,\W} \in \itv[co]{k - L}{k + 1 + L} \right\} \) \) \eqpunct{,} \]
 where $\mathcal{W} = T^1 B(\pi(x), R_\varepsilon) \supset T^1 B(\pi(x), R_1)$.
 If we denote by $C_5 = C_3 C_4 \card G_2$ and
 \[ B = \sum_{k \ge 1} k \sum_{\gamma \in \Gamma_{\wt{W},k} \cap S_2} e^{\int_x^{\gamma x} \wt{F}} \]
 which is finite since $S_2$ is finite, we get
 \begin{align*}
  \sum_{\gamma \in \Gamma_{\wt{W}}} d(x, \gamma x) &e^{\int_x^{\gamma x} \wt{F}} \\
     &\le B + \sum_{k \ge 1} k \sum_{\gamma \in \Gamma_{\wt{W},k} \setminus S_2} e^{\int_x^{\gamma x} \wt{F}} \\
     &\le B + \sum_{k \ge 1}^{k_0} k M_k + C_5 \sum_{k \ge 1} k m_F\( \left\{ \tau_{\varepsilon,\mathcal{W}} \in \itv[co]{k - L}{k + 1 + L} \right\} \cap \mathcal{W}_\varepsilon \) \\
     &\le B + \sum_{k \ge 1}^{k_0} k M_k + C_5 (2 L + 1) \sum_{k \ge 1} k m_F\( \left\{ \tau_{\varepsilon,\mathcal{W}} \in \itv[co]{k}{k + 1} \right\} \cap \mathcal{W}_\varepsilon \) \eqpunct{.}
 \end{align*}
 Since $\mathcal{W}$ contains $T^1 B(\pi(x), R_1)$, it meets the nonwandering set $\Omega$, hence $m_F(\mathcal{W}) > 0$ and we can apply Kac's Lemma \ref{lem:kac} to finally obtain
 \[ \sum_{\gamma \in \Gamma_{\wt{W}}} d(x, \gamma x) e^{\int_x^{\gamma x} \wt{F}} \le B + \sum_{k \ge 1}^{k_0} k M_k + C_5 (2 L + 1) m_F(T^1 M) < +\infty \eqpunct{.} \]
 Therefore $(\Gamma, \wt{F})$ is positive recurrent relatively to $\wt{W}$.
\end{proof}


\section{An intermediate technical criterion of positive recurrence in the universal cover}

\label{sec:fin2}

In this section, we prove a slightly modified version of the preceding criterion (Theorem \ref{thm:crit1}), which will allow us to prove Theorem \ref{thm:crit3} in the next section.
We introduce a notion of $(\Gamma, \wt{F})$-positive recurrence with multiplicity $N \ge 1$, and prove that it is still equivalent to the finiteness of $m_F$.

For an open relatively compact set $\wt{W} \subset \wt{M}$, and $N \ge 1$, define
\[ \Gamma_{\wt{W}}^\star(N) = \set{\gamma \in \Gamma}{\exists y, y' \in \wt{W}, \card \set{g \in \Gamma \setminus \sing{\id}}{\itv[cc]{y}{\gamma y'} \cap g \wt{W} \neq \emptyset} \le N} \eqpunct{.} \]
Of course, $N \mapsto \Gamma_{\wt{W}}^\star(N)$ is increasing, and for all $g \in \Gamma$,
\[ \Gamma_{g \wt{W}}^\star(N) = g \Gamma_{\wt{W}}^\star(N) g^{-1} \eqpunct{.} \]

\begin{theo}[Alternative criterion]
 \label{thm:crit2}
 Let $M$ be a negatively curved orbifold with pinched negative curvature, and $\fn{F}{T^1 M}{\R}$ a H\"older continuous potential with $P(F) < +\infty$.
 Let $m_F$ be its associated Gibbs measure on $T^1 M$.
 \begin{enumerate}
  \item[$(i)$\:] If $F$ is recurrent, and if there exists an open relatively compact subset $\wt{W}$ of $\wt{M}$ meeting $\pi(\wt{\Omega})$ such that for some $x \in \wt{M}$,
   \[ \sum_{\gamma \in \Gamma_{\wt{W}}^\star(N)} d(x, \gamma x) e^{\int_x^{\gamma x} \wt{F} - P(F)} < +\infty \eqpunct{,} \]
   with $N \ge N_{\wt{W}} = 2 \card \set{g \in \Gamma \setminus \sing{\id}}{\overline{\wt{W}} \cap \overline{g \wt{W}} \neq \emptyset} + 1$, then $m_F$ is finite.
  \item[$(ii)$\:] If $m_F$ is finite, then $F$ is recurrent, and for all $x \in \wt{M}$ and $N \ge 1$, we have
   \[ \sum_{\gamma \in \Gamma_{\wt{W}}^\star(N)} d(x, \gamma x) e^{\int_x^{\gamma x} \wt{F} - P(F)} < +\infty \eqpunct{,} \]
   for every open relatively compact subset $\wt{W}$ of $\wt{M}$ large enough to contain a ball $B(x_0, R)$ with $x_0 \in \pi(\wt{\Omega})$ and $R > R_0(x_0)$.
   Moreover, the map $\fn{R_0}{\wt{M}}{\R_+}$ appearing in this statement is bounded on compact sets and satisfies $R_0 = 0$ when $\Lambda(\Gamma) = \partial_\infty \wt{M}$.
 \end{enumerate}
\end{theo}

Like in the previous sections, we can assume that $P(F) = 0$, and we will prove Theorem \ref{thm:crit2} as an immediate consequence of the following two lemmas.

\begin{lemm}
 Let $\wt{W}$ be an open relatively compact subset of $\wt{M}$.
 If $(\Gamma, \wt{F})$ is positive recurrent relatively to $\wt{W}$, then
 \[ \forall z \in \wt{M}, \forall N \ge 1, \sum_{\gamma \in \Gamma_{\wt{W}}^\star(N)} d(z, \gamma z) e^{\int_z^{\gamma z} \wt{F}} < +\infty \eqpunct{.} \]
\end{lemm}

\begin{proof}
 We shall show the convergence of this series by induction over $N \ge 1$.

 If $\gamma \in \Gamma_{\wt{W}}^\star(1)$, then there are $y, y' \in \wt{W}$ such that $\itv[cc]{y}{\gamma y'}$ only meets $\wt{W}$ and $\gamma \wt{W}$, which naturally ensures that $\gamma \in \Gamma_{\wt{W}}$.
 Therefore $\Gamma_{\wt{W}}^\star(1) \subset \Gamma_{\wt{W}}$, and the positive recurrence of $(\Gamma, \wt{F})$ relatively to $\wt{W}$ implies the convergence of the series for $N = 1$.

 We now assume that $(\Gamma, \wt{F})$ is positive recurrent relatively to $\wt{W}$, and that the sum converges for some $N \ge 1$.
 Let $\gamma \in \Gamma_{\wt{W}}^\star(N + 1) \setminus \Gamma_{\wt{W}}^\star(N)$, and pick $y, y' \in \wt{W}$ such that $\itv[cc]{y}{\gamma y'}$ intersects $g_0 \wt{W} = \wt{W}, g_1 \wt{W}, \hdots g_{N+1} \wt{W} = \gamma \wt{W}$ where the $g_i$ are distinct.
 Suppose $\wt{W} \subset B(z, R)$ for some $R > 0$.
 Except for possibly finitely many $\gamma$, one has $d(z, \gamma z) \ge 12 R$, so that $\overline{\wt{W}} \cap \overline{\gamma \wt{W}} = \emptyset$.

 Let
 \[ I_\gamma = \set{i \in \left\{ 1, \hdots, N \right\}}{d(z, g_i z) \ge 2 R \eqand d(\gamma z, g_i z) \ge 2 R} \eqpunct{.} \]
 We shall first treat the case where $I_\gamma \neq \emptyset$.
 Pick some $i \in I_\gamma$ and $w \in \itv[cc]{y}{\gamma y'} \cap g_i \wt{W}$.
 Then $d(y, w) > 0$, $d(w, \gamma y') > 0$, and there are at most $N$ copies $g_j \wt{W}$ that will intersect $\itv[cc]{y}{w} \cup \itv[cc]{w}{\gamma y'}$, since they do not intersect respectively $\gamma \wt{W}$ and $\wt{W}$ by definition of $i \in I_\gamma$.
 Therefore $g_i \in \Gamma_{\wt{W}}^\star(N)$ and $\gamma g_i^{-1} \in \Gamma_{g_i \wt{W}}^\star(N)$, or in other words $g_i^{-1} \gamma \in \Gamma_{\wt{W}}^\star(N)$.
 Moreover, we have either
 \[ d(z, g_i z) \ge d(y, w) - 2 R \ge \frac{d(y, \gamma y')}{2} - 2 R \ge \frac{d(z, \gamma z)}{2} - 3 R \ge \frac{d(z, \gamma z)}{4} \eqpunct{,} \]
 or similarly
 \[ d(z, g_i^{-1} \gamma z) = d(g_i z, \gamma z) \ge \frac{d(z, \gamma z)}{4} \eqpunct{.} \]
 Assume the former happens, for the second case can be treated similarly.
 Remembering that $d(g_i z, \gamma z) \ge 2 R$, and thanks to Lemma \ref{lem:potcontrol}, we get
 \begin{align*}
  d(z, \gamma z) e^{\int_z^{\gamma z} \wt{F}} &\le 4 e^C d(z, g_i z) e^{\int_z^{g_i z} \wt{F}} e^{\int_{z}^{g_i^{-1} \gamma z} \wt{F}} \\
   &\le \frac{2 e^C}{R} d(z, g_i z) e^{\int_z^{g_i z} \wt{F}} d(z, g_i^{-1} \gamma z) e^{\int_{z}^{g_i^{-1} \gamma z} \wt{F}} \eqpunct{.}
 \end{align*}
 For every $\gamma$ such that $I_\gamma \neq \emptyset$, we can find some $g_\gamma = g_i \in \Gamma_W^\star(N)$ such that this estimate (or its symmetric version for $g_i^{-1} \gamma$) holds.
 It follows that
 \begin{align*}
  \sum_{\substack{\gamma \in \Gamma_{\wt{W}}^\star(N+1) \\ I_\gamma \neq \emptyset}} d(z, \gamma z) e^{\int_z^{\gamma z} \wt{F}} &\le \frac{2 e^C}{R} \sum_{\substack{\gamma \in \Gamma_{\wt{W}}^\star(N+1) \\ I_\gamma \neq \emptyset}} d(z, g_\gamma z) e^{\int_z^{g_\gamma z} \wt{F}} d(z, g_\gamma^{-1} \gamma z) e^{\int_z^{g_\gamma^{-1} \gamma z} \wt{F}} \\
  &\le \frac{2 e^C}{R} \sum_{g \in \Gamma_{\wt{W}}^\star(N)} d(z, g z) e^{\int_z^{g z} \wt{F}} \sum_{h \in \Gamma_{\wt{W}}^\star(N)} d(z, h z) e^{\int_z^{h z} \wt{F}} \eqpunct{,}
 \end{align*}
 and this upper bound is finite by the recurrence hypothesis.
 To go from the second to the third line above, observe that in the above reasoning $\gamma \in \Gamma^\star_W(N+1)$ can be written $\gamma = g h$, so that for a given pair $g, h \in \Gamma^\star_W(N)$, there is at most one $\gamma = g h$ in the left sum.

 Now assume that $I_\gamma = \emptyset$.
 Let $G = \set{g \in \Gamma}{d(x, g x) \le 2 R}$.
 Reasoning as in the proof of Lemma \ref{lem:recorbitlift}, we can find $g \in G$, $g' \in \gamma G \gamma^{-1}$, $w \in g \wt{W}$ and $w' \in g' \wt{W}$ such that
 \[ h \wt{W} \cap \itv[cc]{w}{w'} \neq \emptyset \Rightarrow \overline{h \wt{W}} \cap \overline{g \wt{W}} \neq \emptyset \eqor \overline{h \wt{W}} \cap \overline{g' \wt{W}} \neq \emptyset \eqpunct{,} \]
 which means that $g' g^{-1} \in \Gamma_{g \wt{W}}$, or in other words that $\gamma' = g^{-1} g' \in \Gamma_{\wt{W}}$.
 Moreover, note that
 \[ d(z, \gamma' z) = d(g z, g' z) \ge d(z, \gamma z) - 4 R \ge 8 R \eqpunct{.} \]
 Applying once again Lemma \ref{lem:potcontrol}, we get the existence of a constant $C$ (depending on $R$) such that
 \[ d(z, \gamma z) e^{\int_z^{\gamma z} \wt{F}} \le (4 R + d(g z, g' z)) e^{C + \int_{g z}^{g' z} \wt{F}} \le \frac{3}{2} e^C d(z, \gamma' z) e^{\int_z^{\gamma' z} \wt{F}} \eqpunct{.} \]
 This also ensures that the series
 \[ \sum_{\gamma \in \Gamma_{\wt{W}}^\star(N+1), I_\gamma = \emptyset} d(z, \gamma z) e^{\int_z^{\gamma z} \wt{F}} \]
 converges.
 Combining those two cases together, we get the convergence of the series for $N + 1$.
\end{proof}

\begin{lemm}
 For every $\wt{W} \subset \wt{M}$ open relatively compact, there exists $N_{\wt{W}} \ge 1$ such that if
 \[ \exists N \ge N_{\wt{W}}, \exists z \in \wt{W}, \sum_{\gamma \in \Gamma_{\wt{W}}^\star(N)} d(z, \gamma z) e^{\int_z^{\gamma z} \wt{F}} < +\infty \eqpunct{,} \]
 then $(\Gamma, \wt{F})$ is positive recurrent relatively to $\wt{W}$.
\end{lemm}

\begin{proof}
 Set
 \[ N_{\wt{W}} = 2 \card \set{g \in \Gamma \setminus \sing{\id}}{\overline{\wt{W}} \cap \overline{g \wt{W}} \neq \emptyset} + 1 \eqpunct{.} \]
 Let $\gamma \in \Gamma_{\wt{W}}$.
 If $y, y' \in \wt{W}$ are such that
 \[ \itv[cc]{y}{\gamma y'} \cap g \wt{W} \neq \emptyset \Rightarrow \overline{\wt{W}} \cap \overline{g \wt{W}} \neq \emptyset \eqor \overline{\gamma \wt{W}} \cap \overline{g \wt{W}} \neq \emptyset \eqpunct{,} \]
 then clearly at most $N_{\wt{W}}$ copies $g \wt{W}$ with $g \neq \id$ can meet $\itv[cc]{y}{\gamma y'}$, thus $\gamma \in \Gamma_{\wt{W}}^\star(N_{\wt{W}})$.
 We just showed that $\Gamma_{\wt{W}} \subset \Gamma_{\wt{W}}^\star(N_{\wt{W}})$, and we always have $\Gamma_{\wt{W}}^\star(N_{\wt{W}}) \subset \Gamma_{\wt{W}}^\star(N)$ when $N \ge N_{\wt{W}}$, so the result is proved.
\end{proof}


\section{Positive recurrence for the geodesic flow on the manifold}

\label{sec:fin3}

In this section, we shall prove Theorem \ref{thm:crit3}.
We will assume once again that $P(F) = 0$.

\begin{lemm}
 \label{lem:positive-rec-implies-positive-rec-periodic}
 Let $W \subset M$ be open relatively compact, $\W = T^1 W$ and fix any open relatively compact lift $\wt{W}$ of $W$ to $\wt{M}$, so that $\wt{\W} = T^1 \wt{W}$.
 There exists a constant $C \ge 1$ such that, for all $z \in \wt{M}$ and $N \ge 1$, if
 \[ \sum_{\gamma \in \Gamma_{\wt{W}}^\star(C N)} d(z, \gamma z) e^{\int_z^{\gamma z} \wt{F}} < +\infty \eqpunct{,} \]
 then
 \[ \sum_{\substack{p \in \P_\W' \\ n_\W(p) \le N}} l(p) e^{\int_p F} < +\infty \eqpunct{.} \]
\end{lemm}

\begin{proof}
 According to Lemma \ref{lem:crossingcompare}, there exists $C \ge 1$ which depends only on $\wt{\W}$ such that
 \[ \forall p \in \P', \frac{1}{C} n_{\wt{\W}}(p) \le n_\W(p) \le C n_{\wt{\W}}(p) \eqpunct{.} \]
 Pick $p \in \P_\W'$ with $n_\W(p) \le N$ and $\gamma_p \in \Gamma_h$ an hyperbolic element in the conjugacy class associated with $p$.
 Then
 \[ 0 < \frac{1}{C} \le \frac{1}{C} n_\W(p) \le n_{\wt{\W}}(\gamma_p) = n_{\wt{\W}}(p) \le C n_\W(p) \le C N \eqpunct{.} \]
 In particular, one must have at least $n = n_{\wt{\W}}(\gamma_p) \ge 1$, so fix $z_p \in \wt{\pi}(A_{\gamma_p}) \cap \wt{W}$.
 As $\gamma_p$ is primitive, the geodesic segment $\itv[cc]{z_p}{\gamma_p(z_p)}$ will meet at most $n$ copies $g \wt{W}$ where $g \neq \id$, for each of them yields a distinct conjugate $g \gamma_p g^{-1}$ whose axis meets $\wt{\W}$.
 This ensures that $\gamma_p \in \Gamma_{\wt{W}}^\star(C N)$, and therefore
 \[ \sum_{\substack{p \in \P_\W' \\ n_\W(p) \le N}} l(p) e^{\int_p F} \le \sum_{\gamma_p \in \Gamma_h \cap \Gamma_{\wt{W}}^\star(C N)} d(z_p, \gamma_p z_p) e^{\int_{z_p}^{\gamma_p z_p} \wt{F}} \eqpunct{.} \]

 Finally, since $z_p \in \wt{W}$, which is relatively compact, there exists a $R \ge 0$ such that $d(z_p, z) \le R$ for any $p$.
 In particular,
 \[ d(z_p, \gamma_p z_p) \le 2 R + d(z, \gamma_p z) \le 2 d(z, \gamma_p z) \]
 whenever $\gamma_p $ does not belong to the finite set $S = \set{\gamma \in \Gamma}{d(z, \gamma z) < \frac{R}{2}}$.
 Moreover, Lemma \ref{lem:potcontrol} gives the existence of a constant $C'$ such that
 \[ \int_{z_p}^{\gamma_p z_p} \wt{F} \le C' + \int_z^{\gamma z} \wt{F} \eqpunct{.} \]
 Therefore,
 \[ \sum_{\substack{p \in \P_\W' \\ n_\W(p) \le N}} l(p) e^{\int_p F} \le A + 2 e^{C'} \sum_{\gamma \in \Gamma_{\wt{W}}^\star(C N)} d(z, \gamma z) e^{\int_z^{\gamma z} \wt{F}} \eqpunct{,} \]
 where $A$ is the finite sum of terms that correspond to elements $\gamma_p \in S$.
 This concludes the proof.
\end{proof}

\begin{lemm}
 \label{lem:positive-rec-periodic-implies-positive-rec}
 Let $W$ be an open relatively compact subset of $M$ such that $T^1 W$ meets $\Omega$.
 For every open relatively compact lift $\wt{W} \subset \wt{M}$ of $W$ to $\wt{M}$, and for every $\wt{W}'$ open and relatively compact subset of $\wt{W}$, there exists $K \ge 0$ such that for all $N \ge 1$, if
 \[ \sum_{\substack{p \in \P_{\W'}' \\ n_\W(p) \le N + K}} l(p) e^{\int_p F} < +\infty \eqpunct{,} \]
 with $\W' = \Pr(T^1 \wt{W}') \subset T^1 W$, then
 \[ \forall x \in \wt{M}, \sum_{\gamma \in \Gamma_{\wt{W}}^\star(N)} d(x, \gamma x) e^{\int_x^{\gamma x} \wt{F}} < +\infty \eqpunct{.} \]
\end{lemm}

\begin{proof}
 By a routine application of Lemma \ref{lem:potcontrol}, we may assume that the base point $x$ lies in $\wt{W} \cap \pi(\wt{\Omega})$.

 Take $R > 0$ such that $\wt{W} \subset B(x, R)$.
 Since $\wt{W}'$ is relatively compact in $\wt{W}$, $\varepsilon = d(\wt{W}', \wt{M} \setminus \wt{W}) > 0$.
 Lemma \ref{lem:lrsa} gives then the existence of $\rho_\varepsilon = \rho_1(R, R) - \log \varepsilon$ such that for every $y, y', z, z' \in \wt{W}$, if $g \in \Gamma$ satisfies
 \[ \itv[cc]{y}{y'} \cap g \wt{W} \neq \emptyset \eqand \itv[cc]{z}{z'} \cap g \wt{W} = \emptyset \eqpunct{,} \]
 then either $d(y, g \wt{W}) \le \rho_\varepsilon$, or $d(y', g \wt{W}) \le \rho_\varepsilon$, or $g \wt{W}'$ does not meet $\itv[cc]{z}{z'}$.
 In particular, pick $\gamma \in \Gamma_{\wt{W}}^\star(N)$, and $y, y' \in \wt{W}$ such that $\itv[cc]{y}{\gamma y'}$ meets at most $N$ copies $g \wt{W}$, with $g \neq \id$.
 Denote by $K$ the cardinal
 \[ K = 2 \card \set{g \in \Gamma}{d(x, g x) \le \rho_\varepsilon + R} \eqpunct{,} \]
 which depends on both $\wt{W}$ and $\wt{W}'$, but not on $F$, $y$, $y'$ and $\gamma$.
 Fix $z, z' \in \wt{W}$.
 Then $\itv[cc]{z}{\gamma z'}$ can meet $h \wt{W}'$ only if either $h \wt{W}$ meets $\itv[cc]{y}{\gamma y'}$, or
 \[ d(x, h x) \le d(x, y) + d(y, h \wt{W}) \le \rho_\varepsilon + R \eqor d(\gamma x, h x) \le \rho_\varepsilon + R \eqpunct{.} \]
 Therefore there are at most $N + K$ copies $h \wt{W}'$ which meet any segment $\itv[cc]{z}{z'}$ going from $\wt{W}$ to $\gamma \wt{W}$, and $K$ only depends on $\wt{W}$ and $\wt{W}'$.

 According to Lemma \ref{lem:forcehyper}, there exist $g_1, \hdots, g_k \in \Gamma$ and a finite set $S \subset \Gamma$ such that for every $\gamma \in \Gamma \setminus S$, there are $i, j$ such that
$\gamma' = g_j^{-1} \gamma g_i$ is hyperbolic and its axis $\wt{\pi}(A_{\gamma'})$ meets $\wt{W}'$.
 Moreover, the previous discussion ensures that if $\gamma \in \Gamma_{\wt{W}}^\star(N)$ then $\gamma' \in \Gamma_{\wt{W}'}^\star(N + K)$, where the segment going from $\wt{W}'$ to $\gamma' \wt{W}'$
meeting at most $N + K$ copies $h \wt{W}'$ can be chosen as $\itv[cc]{z_{\gamma'}}{\gamma' z_{\gamma'}}$ where $z_{\gamma'} \in \wt{W}' \cap \wt{\pi}(A_{\gamma'})$.
 Denote by $R_1 = \sup \set{d(x, g_i x)}{i = 1, \hdots, k}$, so that
 \[ \left| d(g_j^{-1} x, \gamma' g_i^{-1} x) - d(z_{\gamma'}, \gamma' z_{\gamma'}) \right| \le d(x, g_i x) + d(x, g_j x) + 2 d(x, z_{\gamma'}) \le 2 (R_1 + R) \eqpunct{.} \]
 If $C_1$ is the constant given by Lemma \ref{lem:potcontrol} for $R_2 = 2(R_1 + R)$, then
 \begin{align*}
  \sum_{\gamma \in \Gamma_{\wt{W}}^\star(N)} d(x, \gamma x) e^{\int_x^{\gamma x} \wt{F}} &\le A + \sum_{i,j} \sum_{\gamma' \in \Gamma_{\wt{W}'}^\star(N + K) \cap \Gamma_h} d(g_j^{-1} x, \gamma' g_i^{-1} x) e^{\int_{g_j^{-1} x}^{\gamma' g_i^{-1} x} \wt{F}} \\
   &\le A + k^2 \sum_{\gamma' \in \Gamma_{\wt{W}'}^\star(N + K) \cap \Gamma_h} (R_2 + d(z_{\gamma'}, \gamma' z_{\gamma'})) e^{C + \int_{z_{\gamma'}}^{\gamma' z_{\gamma'}} \wt{F}} \eqpunct{,}
 \end{align*}
 where $A$ is the sum over the elements $\gamma \in S$.
 By possibly adding finitely many elements to $S$, we may assume that every $\gamma \not\in S$ satisfies $d(x, \gamma' x) \ge R_2 + 2 R$ so that
 \[ d(z_{\gamma'}, \gamma' z_{\gamma'}) \ge d(x, \gamma' x) - 2 R \ge R_2 \eqpunct{.} \]
 Therefore
 \[ \sum_{\gamma \in \Gamma_{\wt{W}}^\star(N)} d(x, \gamma x) e^{\int_x^{\gamma x} \wt{F}} \le A + 2 k^2 \sum_{\gamma \in \Gamma_{\wt{W}'}^\star(N + K) \cap \Gamma_h} l(\gamma) e^{\int_{p_{\gamma}} F} \eqpunct{,} \]
 where $p_\gamma$ is the periodic orbit associated with $\gamma \in \Gamma_h$.

 Our goal is now to obtain a sum over primitive hyperbolic isometries.
 Recall that $\gamma \in \Gamma_{\wt{W}'}^\star(N+K)$ means that $\itv[cc]{z_\gamma}{\gamma z_\gamma}$ meets $q$ copies $g \wt{W}'$ with $g \neq \id$ and $q \le N + K$.
 Therefore, if such $\gamma$ can be written $\gamma_0^n$ with $\gamma_0 \in \Gamma_h'$ and $n \ge 1$, then $n$ must divide $q$ and furthermore $\gamma_0 \in \Gamma_{\wt{W}'}^\star(\frac{q}{n}) \subset \Gamma_{\wt{W}'}^\star(N + K)$.
 This means that
 \[ \sum_{\gamma \in \Gamma_{\wt{W}'}^\star(N + K) \cap \Gamma_h} l(\gamma) e^{\int_{p_{\gamma}} F} \le \sum_{1 \le n \le N + K} \sum_{\gamma_0 \in \Gamma_{\wt{W}'}^\star(N + K) \cap \Gamma_h'} n l(\gamma_0) e^{n \int_{p_{\gamma_0}} F} \eqpunct{.} \]
 For every $p \in \P_{\W'}'$ with $n_\W(p) \le N + K$, there are exactly $n_{\W'}(p)$ primitive hyperbolic isometries $\gamma \in \Gamma_{\wt{W}'}^\star(N + K) \cap \Gamma_h'$ such that $p_\gamma = p$.
 Therefore
 \begin{align*}
  \sum_{\gamma_0 \in \Gamma_{\wt{W}'}^\star(N + K) \cap \Gamma_h'} l(\gamma_0) e^{\int_{p_{\gamma_0}} F} &= \sum_{\substack{p \in \P_{\W'}' \\ n_{\W'}(p) \le N + K}} n_{\W'}(p) l(p) e^{\int_p F} \\
  &\le (N + K) \sum_{\substack{p \in \P_{\W'}' \\ n_{\W'}(p) \le N + K}} l(p) e^{\int_p F} \eqpunct{,}
 \end{align*}
 which is finite by hypothesis.
 Moreover, $l(\gamma_0) \ge 1$ except for maybe finitely many $\gamma_0$ in the above sum, hence the series
 \[ \sum_{\gamma_0 \in \Gamma_{\wt{W}'}^\star(N + K) \cap \Gamma_h'} e^{\int_{p_{\gamma_0}} F} \le A' + (N + K) \sum_{\substack{p \in \P_{\W'}' \\ n_{\W'}(p) \le N + K}} l(p) e^{\int_p F} \]
 is convergent.
 This implies in particular that there exists a constant $C'$ such that $\int_{p_{\gamma_0}} F \le C'$ for every $\gamma_0 \in \Gamma_{\wt{W}'}^\star(N + K) \cap \Gamma_h'$.
 Gathering all these elements together, we obtain
 \begin{align*}
  \sum_{\gamma \in \Gamma_{\wt{W}'}^\star(N + K) \cap \Gamma_h} l(\gamma) e^{\int_{p_{\gamma}} F} &\le \( \sum_{1 \le n \le N + K} n e^{(n - 1) C'} \) \sum_{\gamma_0 \in \Gamma_{\wt{W}'}^\star(N + K) \cap \Gamma_h'} l(\gamma_0) e^{\int_{p_{\gamma_0}} F} \\
    &\le \frac{(N + K)^3}{2} e^{(N + K - 1) C'} \sum_{\substack{p \in \P_{\W'}' \\ n_{\W'}(p) \le N + K}} l(p) e^{\int_p F} \eqpunct{,}
 \end{align*}
 which proves precisely the statement of this lemma.
\end{proof}

Let us now complete the proof of Theorem \ref{thm:crit3}.

\begin{proof}
 First, if $m_F$ is finite, we know by Theorem \ref{thm:crit2} $(ii)$ that $F$ is recurrent and that $(\Gamma, \wt{F})$ is positive recurrent in the sense of Theorem \ref{thm:crit2} for all $N \ge 1$ and $x \in \wt{M}$.
 By Lemma \ref{lem:positive-rec-implies-positive-rec-periodic}, this implies that $F$ is positive recurrent in the sense of Definition \ref{def:positive-recurrent}.

 Conversely, if $F$ is positive recurrent in the sense of Definition \ref{def:positive-recurrent} for some $N > K = 2 \card \set{g \in \Gamma}{d(x, g x) \le \rho_\varepsilon + R}$, Lemma \ref{lem:positive-rec-periodic-implies-positive-rec} implies that $(\Gamma, \wt{F})$ will be positive recurrent in the sense of Theorem \ref{thm:crit2} for some integer greater than $N - K \ge 1$ and, since $F$ is assumed to be recurrent in the sense of Definition \ref{def:recurrent-potential}, Theorem \ref{thm:crit2} $(i)$ implies that $m_F$ is finite.
\end{proof}


\section{Finiteness of Gibbs measures from equidistribution of weighted periodic orbits}

\label{sec:fin4}


\subsection{An equidistribution result for nonprimitive orbits}

If $\gamma$ is an hyperbolic isometry, denote by $\mathcal{L}_\gamma$ the \name{Lebesgue measure along $\gamma$}, that is the measure defined on $T^1 \wt{M}$ by
\[ \int \wt{\varphi} d\mathcal{L}_\gamma = \int_\R \wt{\varphi}(\wt{g}^t(\wt{v})) dt \]
for any compactly supported continuous function $\wt{\varphi}$ on $T^1 \wt{M}$, where $\wt{v}$ is any vector of $A_\gamma$.
This definition ignores the multiplicity of $\gamma$ : if $\gamma = \gamma_0^k$ with $\gamma_0$ primitive, then $\mathcal{L}_\gamma = \mathcal{L}_{\gamma_0}$.

If $p$ is a periodic orbit of the geodesic flow, denote by $\mathcal{L}_p$ the \name{Lebesgue measure along $p$}, that is the measure defined on $T^1 M$ by
\[ \int \varphi d\mathcal{L}_p = \int_0^{l(p)} \varphi(g^t(v)) dt \]
for any compactly supported continuous function $\varphi$ on $T^1 M$, where $v$ is any vector of $p$.
This definition takes into account the multiplicity of the periodic orbit : if $p$ is the $k$-th iterate of a primitive periodic orbit $p_0$, then $\mathcal{L}_p = k \mathcal{L}_{p_0}$.
In the following, we will denote by $m(p)$ the multiplicity of a periodic orbit $p \in \P$.

We remark that if $\fun{\Pi}{\wt{\mu}}{\mu}$ denotes the projection of locally finite $\Gamma$-invariant measures through the branched cover $T^1 \wt{M} \to T^1 M$, then for every $\gamma_0 \in \Gamma_h$ primitive we have
\[ \Pi\( \sum_{g \in \Gamma} \mathcal{L}_{g^{-1} \gamma_0 g} \) = \mathcal{L}_{p_0} \eqpunct{,} \]
where $p_0$ is the primitive periodic orbit on which the axis of $\gamma_0$ projects.

The third finiteness criterion will be derived from the next equidistribution result for weighted sums of measures supported on nonprimitive periodic orbits of the geodesic flow, which itself requires the Gibbs measure to be mixing.
Note that by virtue of Babillot's theorem (see theorem 1 in \cite{Bab}) this is equivalent to the nonarithmeticity of the length spectrum, regardless whether the Gibbs measure is finite or not.
In more dynamical terms, it is also equivalent to the topological mixing of the geodesic flow.

\begin{theo}
 \label{thm:equidist}
 Let $M = \Gamma \backslash \wt{M}$ be a negatively curved orbifold with pinched negative curvature, topologically mixing geodesic flow, and let $\fn{F}{T^1 M}{\R}$ be a H\"older continuous potential.
 Assume that the pressure $P(F)$ is finite and positive.
 Define
 \[ \nu_{F,t} = P(F) e^{-P(F) t} \sum_{\substack{p \in \P \\ l(p) \le t}} e^{\int_p F} \frac{1}{m(p)} \mathcal{L}_p \eqpunct{.} \]
 \begin{enumerate}
  \item If $m_F$ is finite, then $\nu_{F,t}$ converges weakly to $\frac{m_F}{\| m_F \|}$.
  \item If $m_F$ is infinite, then $\nu_{F,t}$ converges weakly to $0$.
 \end{enumerate}
\end{theo}

\begin{proof}
 According to the remark above, $\nu_{F,t}$ can be rewritten in terms of primitive periodic orbits as
 \[ \nu_{F,t} = \delta e^{-\delta t} \sum_{k \ge 1} \sum_{\substack{p \in \P' \\ k l(p) \le t}} e^{k \int_p F} \mathcal{L}_p \eqpunct{.} \]
 Since $\fun{\Pi}{\wt{\mu}}{\mu}$ is continuous with respect to the weak convergences of measures, and $\Pi(\wt{m}_F) = m_F$, it is enough to study the weak convergence of the sequence
 \[ \nu_{F,t}' = \delta e^{-\delta t} \sum_{k \ge 1} \sum_{\substack{\gamma \in \Gamma_h' \\ k l(\gamma) \le t}} e^{k \int_{z_\gamma}^{\gamma z_\gamma} \wt{F}} \mathcal{L}_\gamma = \delta e^{-\delta t} \sum_{\substack{\gamma \in \Gamma_h \\ l(\gamma) \le t}} e^{\int_{z_\gamma}^{\gamma z_\gamma} \wt{F}} \mathcal{L}_\gamma \eqpunct{,} \]
 where $z_\gamma \in \wt{M}$ is any point on the invariant axis of $\gamma$ on $\wt{M}$.
 The first step of the proof of \cite[Theorem 9.11]{PPS} shows that $\nu_{F,t}'$ weak-star converges to $\frac{\wt{m}_F}{\| \wt{m}_F \|}$ (respectively $0$) whenever $m_F$ is finite (respectively infinite), therefore $\nu_{F,t}$ weak-star converges to $\frac{m_F}{\| m_F \|}$ or $0$ accordingly.
\end{proof}

\begin{theo}
 \label{thm:equidist2}
 Let $M = \Gamma \backslash \wt{M}$ be a negatively curved orbifold with pinched negative curvature, topologically mixing geodesic flow, and let $\fn{F}{T^1 M}{\R}$ be a H\"older continuous potential.
 Assume that the pressure $P(F)$ is finite and positive.
 For $c > 0$, define
 \[ \zeta_{F,c,t} = \frac{P(F)}{1 - e^{-c P(F)}} e^{-P(F) t} \sum_{\substack{p \in \P \\ t - c < l(p) \le t}} e^{\int_p F} \frac{1}{m(p)} \mathcal{L}_p \eqpunct{,} \]
 with the convention that $\frac{\delta}{1 - e^{-c \delta}} = 1$ when $c = 0$.
 \begin{enumerate}
  \item If $m_F$ is finite, then $\zeta_{F,c,t}$ converges weakly to $\frac{m_F}{\| m_F \|}$ for every $c > 0$.
  \item If $m_F$ is infinite, then $\zeta_{F,c,t}$ converges weakly to $0$ for every $c > 0$.
 \end{enumerate}
\end{theo}

\begin{proof}
 First assume that $m_F$ is finite.
 It is enough to show the convergence of this sequence of measures when tested against nonnegative continuous functions with compact support in $T^1 M$.
 Let $\varphi$ be such a test function.
 If $\int \varphi dm_F = 0$, then the support of $\varphi$ does not meet the nonwandering set $\Omega$, so $\zeta_{F,c,t} = 0$ and the result holds.
 Otherwise denote by $\delta = P(F)$, set $\kappa = \max(0, -\delta) + 1$ so that $\delta + \kappa \ge 1$ and $m_{F+\kappa} = m_F$, and let $\P_\varphi = \set{p \in \P}{p \cap \supp \varphi \neq \emptyset}$ as well as
 \[ \func{f}{\P_\varphi}{\itv[co]{0}{+\infty}}{p}{l(p)} \eqand \func{g}{\P_\varphi}{\itv[co]{0}{+\infty}}{p}{\frac{\| m_F \|}{\int \varphi dm_F} e^{\int_p F} \frac{1}{m(p)} \int_p \varphi d\mathcal{L}_p} \eqpunct{.} \]
 With these notations, Theorem \ref{thm:equidist} applied to the potential $F + \kappa$ ensures that
 \[ \lim_{t \to +\infty} (\delta + \kappa) e^{-(\delta + \kappa) t} \sum_{\substack{p \in \P_\varphi \\ f(p) \le t}} e^{\kappa f(p)} g(p) = 1 \eqpunct{.} \]
 Note that $f$ is proper since only finitely many periodic orbits of length smaller than some constant can meet the support of $\varphi$.
 Therefore Lemma 9.5 from \cite{PPS} shows that for every $c > 0$ one has
 \[ \lim_{t \to +\infty} \frac{\delta}{1 - e^{-c \delta}} e^{-\delta t} \sum_{\substack{p \in \P_\varphi \\ t - c < f(p) \le t}} g(p) = 1 \eqpunct{,} \]
 which means exactly that $\int \varphi d\zeta_{F,c,t}$ converges to $\int \varphi dm_F$.

 Now assume that $m_F$ is infinite, and take $\kappa$ as before so that $\delta + \kappa > 0$.
 If $\varphi$ is a nonnegative continuous function with compact support in $T^1 M$, note that $\int \varphi d\zeta_{F,c,t}$ is comparable with
 \begin{align*}
  \frac{\delta}{1 - e^{-c \delta}} \frac{1}{\delta + \kappa} &\( \int \varphi d\nu_{F+\kappa,t} - \int \varphi d\nu_{F+\kappa,t-c} \) \\
    &= \frac{\delta}{1 - e^{-c \delta}} e^{-\delta t} \sum_{\substack{p \in \P \\ t-c < l(p) \le t}} e^{(l(p) - t) \kappa} e^{\int_p F} \frac{1}{m(p)} \int_p \varphi d\mathcal{L}_p \eqpunct{,}
 \end{align*}
 which goes to $0$ as $t$ goes to infinity following Theorem \ref{thm:equidist} $(2)$.
\end{proof}


\subsection{Proof of Theorem \ref{thm:crit4}}

In this section, the assumptions of Theorem \ref{thm:crit4} are satisfied.
The geodesic flow is topologically mixing, and $F$ is a H\"older continuous potential with finite pressure.
For $\W \subset T^1 M$, $c > 0$ and $t \ge 0$, define
\[ Z_{c,t}(F, \W) = \sum_{\substack{p \in \P \\ t - c < l(p) \le t}} n_\W(p) e^{\int_p F - P(F)} \eqpunct{.} \]

\begin{lemm}
 If $m_F$ is finite, then for every open relatively compact set $\W$ meeting $\Omega$ and every $c > 0$ there exist a constant $C > 0$ and $t_0 \ge 0$ such that
 \[ \forall t \ge t_0, Z_{c,t}(F, \W) \ge C \eqpunct{.} \]
\end{lemm}

\begin{proof}
 Fix $c > 0$, and denote by $\delta = P(F)$.
 First, Lemma \ref{lem:crossingcover} ensures that for every $\W' \subset \W$, there is a constant $C_1 > 0$ such that
 \[ \forall t \ge 0, Z_{t,c}(F, \W') \le C_1 Z_{t,c}(F, \W) \eqpunct{.} \]
 It is therefore enough to show the result when $\pi(\W) = B(\overline{x}, R)$ intersects $\pi(\Omega)$ but $R$ is small enough that the $\Gamma$-images of $B(x, R)$ in $\wt{M}$ are pairwise disjoint.
 Let $\fn{\varphi}{T^1 M}{\R}$ be a continuous map with compact support in $\W$, such that $\int \varphi dm_F > 0$ and $0 \le \varphi \le 1$.
 Since $\pi(\W)$ is a small ball, the intersection of $\W$ with a periodic orbit $p$ of the geodesic flow is a collection of at most $n_\W(p)$ geodesic segments of length smaller than the diameter of $\W$, each of them being visited $m(p)$ times.
 Hence
 \[ \forall p \in \P, \int \varphi d\mathcal{L}_p \le n_\W(p) m(p) 2 R \eqpunct{.} \]
 According to Theorem \ref{thm:equidist2} $(1)$, the finiteness of $m_F$ gives that
 \[ \lim_{t \to +\infty} e^{-\delta t} \sum_{\substack{p \in \P \\ t - c < l(p) \le t}} e^{\int_p F} \frac{1}{m(p)} \int \varphi d\mathcal{L}_p = \frac{1 - e^{-c \delta}}{\delta \| m_F \|} \int \varphi dm_F = C_2 > 0 \eqpunct{.} \]
 Therefore, with $\delta = P(F)$, there is a $t_0 \ge 0$ such that
 \begin{equation*}
  \forall t \ge t_0, \frac{C_2}{2} e^{-c \delta} \le \sum_{\substack{p \in \P \\ t - c < l(p) \le t}} e^{\int_p F - P(F)} \frac{1}{m(p)} \int \varphi d\mathcal{L}_p \le 2 R Z_{t,c}(F, \W) \eqpunct{.} \qedhere
 \end{equation*}
\end{proof}

\begin{lemm}
 If $m_F$ is finite, then for every open relatively compact set $\W$ intersecting $\Omega$ and every $c > 0$ there is a constant $C > 0$ such that
 \[ \forall t \ge 0, Z_{t,c}(F, \W) \le C \eqpunct{.} \]
\end{lemm}

\begin{proof}
 Since the set $\Sigma$ of singular points of $M$ is locally finite, we can cover $\W$ by a finite collection of open relatively compact sets $\W_i = T^1 B(\overline{x}_i, R_i)$ for which there is an $\varepsilon > 0$ such that for each $i$ we have
 \begin{itemize}
  \item either $x_i \not\in \wt{\Sigma}$ and the $\Gamma$-images of $B(x_i, R_i + \varepsilon)$ in $\wt{M}$ are pairwise disjoint, in which case we let $s_i = 1$ ;
  \item or $x_i \in \wt{\Sigma}$ and $B(x_i, R_i + \varepsilon) \cap g B(x_i, R_i + \varepsilon)$ if and only if $g$ is an elliptic isometry fixing $x_i$, in which case we denote by $s_i$ the cardinal of the stabilizer of $x_i$.
 \end{itemize}
 Let $S = \max s_i < +\infty$.
 Lemma \ref{lem:crossingcover} gives then a constant $C_1$ such that
 \[ \forall p \in \P, n_\W(p) \le C_1 \sum_i n_{\W_i}(p) \eqpunct{.} \]
 It is therefore enough to find an upper bound when $\W = T^1 B(\overline{x}_i, R_i)$ for every $i$.
 Let $\fn{\varphi}{T^1 M}{\R}$ be a continuous map with compact support such that $\varphi = 1$ on $T^1 B(x_i, R_i + \varepsilon)$.
 By definition of $x_i$ and $R_i$, the intersection of $T^1 B(x_i, R_i + \varepsilon)$ with a periodic orbit $p$ that meets $\W$ is a collection of at least $n_\W(p)$ geodesic segments of length greater than $2 \varepsilon$, each of them being visited $m(p)$ times.
 However, at most $s_i$ copies of these geodesic segments in the universal cover are going to project onto the same geodesic segment of $M$.
 Hence
 \[ \forall p \in \P, \int \varphi d\mathcal{L}_p \ge \frac{n_\W(p) m(p)}{s_i} 2 \varepsilon \ge \frac{n_\W(p) m(p)}{S} 2 \varepsilon \eqpunct{.} \]
 According to Theorem \ref{thm:equidist2} $(1)$, the finiteness of $m_F$ gives that
 \[ \lim_{t \to +\infty} e^{-\delta t} \sum_{\substack{p \in \P \\ t - c < l(p) \le t}} e^{\int_p F} \frac{1}{m(p)} \int \varphi d\mathcal{L}_p = \frac{1 - e^{-c \delta}}{\delta \| m_F \|} \int \varphi dm_F < +\infty \eqpunct{.} \]
 Therefore there is a $C_2 > 0$ such that
 \begin{equation*}
  \forall t \ge 0, \frac{2 \varepsilon}{S} Z_{t,c}(F, \W) \le \sum_{\substack{p \in \P \\ l(p) \le t}} e^{\int_p F - P(F)} \frac{1}{m(p)} \int \varphi d\mathcal{L}_p \le e^{c \delta} C_2 \eqpunct{.} \qedhere
 \end{equation*}
\end{proof}

The same proof using Theorem \ref{thm:equidist2} $(2)$ instead of $(1)$ yields immediately the following result.

\begin{lemm}
 If $m_F$ is infinite, then for every $\W$ open relatively compact meeting $\Omega$ and every $c > 0$ we have
 \[ \lim_{t \to +\infty} Z_{t,c}(F, \W) = 0 \eqpunct{.} \]
\end{lemm}


\section{Applications}

\label{sec:examples}

In this section, we show that our results let us retrieve some partial finiteness criteria existing in the litterature.
It is likely that our criteria will allow to find new interesting examples where the Bowen-Margulis measure is finite, in addition to known examples of \cite{Pei}, \cite{Ancona}.


\subsection{Finiteness of Gibbs measures on geometrically finite manifolds}

Geometrically finite manifolds are the most simple negatively curved manifolds with infinite volume.
We refer to \cite{Bowditch} for details on their geometry.
Recall simply that they can be decomposed into the union of a compact part, finitely many finite volume ends, the \name{cusps}, and finitely many infinite volume ends, the \name{funnels}.

Moreover, the nonwandering set $\Omega \subset T^1 M$ of such a manifold is (transversally) a Cantor set completely included in the unit tangent bundle of the compact part and of the cusps.
Therefore, dynamically, we can forget the funnels.

In \cite{DOP}, Dal'bo-Otal-Peign\'e proposed a finiteness criterion for the so-called \name{Bowen-Margulis measure} of such manifolds which is, in our terminology, the Gibbs measure $m_0$ associated with the zero potential $F \equiv 0$.
It has been generalized in \cite{Cou} to all Gibbs measures.
We refer to \cite{PPS}, where it was also exposed, because the notations and construction of \cite{Cou} are slightly different (although equivalent).

Recall that for such manifolds, the lift to the universal cover of a cusp is called a \name{horoball}, and the stabilizer of a horoball in $\Gamma$ is a \name{parabolic subgroup}, denoted by $\Pi$.

Let $\fn{F}{T^1 M}{\R}$ be a H\"older continuous potential, and $\wt{F}$ be its $\Gamma$-invariant lift to $T^1 \wt{M}$.
Recall that the pressure $P(F)$ that we consider since the beginning is the critical exponent of the Poincar\'e series
\[ \sum_{\gamma \in \Gamma} e^{\int_x^{\gamma x} (\wt{F} - s)} \eqpunct{,} \]

The finiteness criterion from Dal'bo-Otal-Peign\'e-Coud\`ene can be stated in this way.

\begin{theo}[\cite{DOP}, \cite{Cou}, \cite{PPS}]
 Let $M$ be a negatively curved geometrically finite orbifold with pinched negative curvature, and $\fn{F}{T^1 M}{\R}$ a H\"older continuous potential with finite pressure.
 Assume that $(\Gamma, F)$ is divergent (i.e. $F$ recurrent).
 Then the Gibbs measure $m_F$ is finite if and only if the series
 \[ \sum_{\pi \in \Pi} d(x, \pi x) e^{\int_x^{\pi x} (\wt{F} - P(F))} \]
 converges for every parabolic subgroup $\Pi < \Gamma$.
\end{theo}

We will infer immediately this result from the following propostion.

\begin{prop}
 The convergence of the series above is exactly given by our criterion of positive recurrence of definition \ref{def:pos-rec-bis} applied to a connected lift $\wt{W}$ inside a fundamental domain of the compact part of $M$.
\end{prop}

\begin{proof}
 As usual, we may assume that $P(F) = 0$.
 Let $\wt{W}$ be the lift to the universal cover of this compact part into a connected domain inside a Dirichlet fundamental domain for the action of $\Gamma$ on $\wt{M}$.
 We may assume that the frontier of $\wt{W}$ is the union of finitely many submanifolds of $\wt{M}$ of codimension $1$, so that the intersection of any geodesic segment with this frontier is a finite union of segments.

 Now recall the definition of $\Gamma_{\wt{W}}$ : it is the set of elements $\gamma$ such that, for some $y, y' \in \wt{W}$, the geodesic $\itv[cc]{y}{\gamma y'}$ will intersect $\Gamma \wt{W}$ only at the beginning and at the end.
 More precisely, this geodesic intersects $g \wt{W}$ if and only if $g \overline{\wt{W}} \cap \overline{\wt{W}} \neq \emptyset$ or $g \overline{\wt{W}} \cap \gamma \overline{\wt{W}} \neq \emptyset$.
 Geometrically on $M = \Gamma \backslash \wt{M}$, it means that except during a time bounded from above by $\diam(W)$ at the beginning and at the end, and maybe except for finitely many $\gamma$, the projection $\Pr(\itv[cc]{y}{\gamma y'})$ has to leave the compact part in the middle.
 This is possible only if it enters inside some cusp $C$ and then returns in the compact part.

 Lift this cusp to a horoball $\wt{C}$ which has a common boundary with $\wt{W}$ (and therefore $\gamma \wt{W}$).
 Without loss of generality, we can assume that both $y$ and $y'$ lie on this common boundary $\partial \wt{C} \cap \partial \wt{W}$.
 Therefore, there exists some parabolic element $\pi$ inside the parabolic subgroup $\Pi$ stabilizing $\wt{C}$ such that $\pi y' \in \partial \wt{C} \cap \partial \gamma \wt{W}$.
 Except maybe during a bounded length (less than $2 \diam \wt{W}$), the geodesic segments $\itv[cc]{y}{\gamma y'}$ and $\itv[cc]{y}{\pi y'}$ stay uniformly close (at distance less than $\diam \wt{W}$), uniformly in $\pi \in \Pi$, so that $\left| \int_x^{\gamma x} \wt{F} - \int_x^{\pi x} \wt{F} \right|$ stay uniformly bounded by a constant depending only on $F$ and $\diam \wt{W}$.

 Thus, the series $\sum_{\gamma \in \Gamma_{\wt{W}}} d(x, \gamma x) e^{\int_x^{\gamma x} \wt{F}}$ is bounded from above, up to some constant, by the sum over the finite number of parabolic subgroups stabilizing a horoball with a common frontier with $\wt{W}$, of the sum $\sum_{p \in \Pi} d(x, \pi x) e^{\int_x^{\pi x} \wt{F}}$.

 Conversely, the same reasoning shows that $\Pi < \Gamma_{\wt{W}}$, so that the reverse inequality is trivial.
 This concludes the proof of the proposition.
\end{proof}


\subsection{Existence of finite Bowen-Margulis measures on geometrically infinite manifolds}

We shall now apply our criteria to recover the finiteness of the Bowen-Margulis measure $m_{BM} = m_0$ (which is the measure of maximal entropy, associated with any potential cohomologous to a constant), for free products of Kleinian groups in Schottky position as exposed in \cite{Pei}.

Recall first the main result in \cite{Pei}.

\begin{theo}[Peign\'e, \cite{Pei}]
 There exist geometrically infinite hyperbolic manifolds of dimension at least $4$ with finite Bowen-Margulis measure $m_{BM} = m_0$.
 These manifolds are constructed by taking the free product of two Kleinian subroups of $\SO^o(4, 1)$ in Schottky position one of them geometrically infinite inside some $\SO^o(3, 1) < \SO^o(4, 1)$, the second divergent with a larger critical exponent.
\end{theo}

Let us explain a little bit its construction.
Let $G$ and $H$ be two Kleinian groups, i.e. two nonelementary discrete torsion-free groups of orientation-preserving isometries of the $d$-dimensional standard hyperbolic space $\H^{d+1}$ ($d \ge 1$) with constant curvature $-1$.
They are said to be in \name{Schottky position} when there exist two closed disjoint sets $F_G, F_H \subset \S^d$ such that
\[ G^\star (\S^d \setminus F_G) \subset F_G \eqand H^\star (\S^d \setminus F_H) \subset F_H \eqpunct{,} \]
where $G^\star = G \setminus \sing{\id_G}$ and $H^\star = H \setminus \sing{\id_H}$.
Note that the limit set $\Lambda(G)$ of $G$ must be a subset of $F_G$ since the fixed points of hyperbolic elements of $G$ can only lie in $F_G$, and likewise $\Lambda(H) \subset F_H$.

A variation of the ping-pong lemma shows that the group $\Gamma$ generated by $G$ and $H$ is equal to the free product $\Gamma = G \star H$ of $G$ and $H$, and that the limit set $\Lambda(\Gamma)$ of $\Gamma$ is included in $F_G \cup F_H$.
Any $\gamma \in \Gamma$ has an unique representation $\gamma = g_1 h_1 \hdots g_n h_n$ where $n \ge 1$, $g_1 \in G$, $g_j \in G^\star$ ($j > 1$), $h_n \in H$ and $h_j \in H^\star$ ($j < n$).
We denote by $\Gamma_G$ the subset of $\gamma \in \Gamma$ which are written $g \gamma'$ with $g \in G^\star$ in this representation, and we define $\Gamma_H$ similarly, so that $\Gamma$ is the disjoint union of $\Gamma = \Gamma_G \sqcup \Gamma_H \sqcup \sing{\id}$.

In the following, we will readily identify $G$ and $H$ with their respective images by the canonical inclusions $G \to G \star H$ and $H \to G \star H$.
We will also denote by $\delta_G$, $\delta_H$ and $\delta_\Gamma = P(0)$ the respective critical exponents of $G$, $H$ and $\Gamma = G \star H$, i.e. the exponents of convergence of the Poincar\'e series
\[ \sum_{g \in G} e^{-s d(x, g x)}, \sum_{h \in H} e^{-s d(x, h x)} \eqand \sum_{\gamma \in \Gamma} e^{-s d(x, \gamma x)} \eqpunct{.} \]

\begin{theo*}[\cite{Pei}, Theorem A]
 Let $\Gamma = G \star H$ be the free product of two Kleinian groups in Schottky position.
 If $\delta_G > \delta_H$ and $G$ is divergent, then $\delta_\Gamma > \max(\delta_G, \delta_H)$, $\Gamma$ is divergent and its Bowen-Margulis measure is finite.
\end{theo*}

These assumptions are satisfied in particular with $G$ a convex cocompact subgroup, which is always divergent, and $H$ a geometrically infinite subgroup of $\SO^o(3, 1) < \SO^o(4, 1)$.
By conjugating $G$ and $H$, it is easy to obtain the inequality $\delta_G > \delta_H$, and the fact that they are in Schottky position.
When $d \ge 2$, one can even assume $H$ to be finitely generated but non geometrically finite.

We will use the criterion from Theorem \ref{thm:crit1} to retrieve the finiteness of the Bowen-Margulis measure in these examples.
To achieve this, we first need the following lemma which will help us to choose the right open relatively compact subset.

\begin{lemm}
 \label{lem:schottkymaxdist}
 For every $x \in \H^{d+1}$, there exists $R > 0$ such that
 \[ \forall g \in \Gamma_G, \forall h \in \Gamma_H, \itv[cc]{g(x)}{h(x)} \cap B(x, R) \neq \emptyset \eqpunct{.} \]
\end{lemm}

\begin{proof}
 Take $U_G$ and $U_H$ two open sets of $\H^{d+1} \cup \S^d$, respectively containing $F_G$ and $F_H$, and whose closures are disjoint.
 In particular, there is a $R > 0$ such that
 \[ \forall \eta \in U_G, \forall \xi \in U_H, d(x, \itv[oo]{\eta}{\xi}) < R \eqpunct{.} \]
 The Schottky position condition ensures that the attractive fixed point of any element of $\Gamma_G$ must lie in $F_G$, therefore the $\Gamma_G$-orbit of $x$ accumulates in $F_G$, and only finitely many $g \in \Gamma_G$ are such that $g(x) \not\in U_G$.
 Hence we may assume that $R$ is large enough that $g(x) \in U_G$ whenever $d(x, g(x)) \ge R$, and likewise that $h(x) \in U_H$ whenever $d(x, h(x)) \ge R$.
 In any case, the geodesic segment $\itv[cc]{g(x)}{h(x)}$ will meet $B(x, R)$.
\end{proof}

We now use notations from Theorem \ref{thm:crit1}, and we take $\wt{W} = B(x, R)$ with a fixed $x$ and $R$ large enough according to Lemma \ref{lem:schottkymaxdist} and Theorem \ref{thm:crit1}.

We will prove the following.

\begin{prop}
 For $\wt{W} = B(x, R)$ with $R$ given by Lemma \ref{lem:schottkymaxdist}, $(\Gamma, \wt{F})$ is positive recurrent, and therefore $m_F$ is finite.
\end{prop}

\begin{lemm}
 There exists $S \subset \Gamma$ finite such that
 \[ \Gamma_{\wt{W}} \setminus S \subset G \cup H \eqpunct{.} \]
\end{lemm}

\begin{proof}
 If $\gamma \not\in G \cup H$, then it can be written $\gamma = g h \gamma'$ or $\gamma = h g \gamma'$, with $g \in G^\star$, $h \in H^\star$ and $h \gamma' \not\in G$ (respectively $g \gamma' \not\in H$).
 We only consider the first case, the second case being similar.
 Note that
 \[ g h \gamma' \not\in \Gamma_{\wt{W}} \Leftrightarrow h \gamma' g \not\in g^{-1} \Gamma_{\wt{W}} g = \Gamma_{g^{-1} \wt{W}} \eqpunct{.} \]
 Now consider the geodesic segment $I = \itv[cc]{g^{-1}(x)}{h \gamma'(x)} = \itv[cc]{g^{-1}(x)}{h \gamma' g g^{-1}(x)}$.
 Both $g^{-1} \wt{W}$ and $h \gamma' \wt{W}$ intersect $I$, but according to Lemma \ref{lem:schottkymaxdist} $\wt{W}$ also intersects $I$.
 Finally, there is only a finite set $S$ of such $\gamma$ for which $g^{-1} \wt{W}$ or $h \gamma' \wt{W}$ meets $\wt{W}$, so that if $\gamma \in \Gamma_{\tilde{W}} \setminus S$, then $\gamma \in G \cup H$.
\end{proof}

We can now complete the proof.
Since $\delta_\Gamma > \max(\delta_G, \delta_H)$ we have
\[ \sum_{g \in G} d(x, g x) e^{- \delta_\Gamma d(x, g x)} < +\infty \eqpunct{,} \]
and the same goes for $H$.
Therefore,
\[ \sum_{\gamma \in G \cup H} d(x, \gamma x) e^{- \delta_\Gamma d(x, \gamma x)} < +\infty \eqpunct{,} \]
and Theorem \ref{thm:crit1} ensures that the Bowen-Margulis measure $m_{BM} = m_0$ of $T^1 M$ is finite.

\section*{Acknowledgements}

V. Pit was partially supported by FAPESP project number 2011/12338-0 and CAPES.
Both authors benefited from the ANR grant ANR JCJC-0108 coordinated by B. Schapira during this work. 

\bibliographystyle{amsalpha}
\bibliography{gibbsfin}

\end{document}